%% file: ExoticMCG.tex
\documentclass[a4paper]{amsart}
\usepackage{amsfonts}

\usepackage{amsmath}
\usepackage{amssymb}
\usepackage{graphicx}
\usepackage{verbatim}
\usepackage{enumerate}
\usepackage{mathtools}
\usepackage{ogonek}

\newtheorem{thm}{Theorem}[section]
\newtheorem{cor}[thm]{Corollary}
\newtheorem{lem}[thm]{Lemma}
\newtheorem{prop}[thm]{Proposition}
\theoremstyle{definition}
\newtheorem{defn}[thm]{Definition}
\theoremstyle{remark}
\newtheorem{rem}[thm]{Remark}

\numberwithin{equation}{section}

\newcommand{\bC}{\mathbb{C}}

\newcommand{\bF}{\mathbb{F}}

\newcommand{\bH}{\mathbb{H}}

\newcommand{\bP}{\mathbb{P}}
\newcommand{\bQ}{\mathbb{Q}}
\newcommand{\bR}{\mathbb{R}}
\newcommand{\bS}{\mathbb{S}}

\newcommand{\bZ}{\mathbb{Z}}

\newcommand{\gC}{\bold{C}}

\newcommand{\MT}[2]{\bold{MT #1}(#2)}
\newcommand{\MTtheta}{\bold{MT \theta}}

\newcommand\lra{\longrightarrow}

\newcommand\trf{\mathrm{trf}}
\newcommand\Diff{\mathrm{Diff}}

\newcommand\Bun{\mathrm{Bun}}
\newcommand\Th{\mathrm{Th}}
\newcommand\colim{\mathrm{colim \,}}

\newcommand{\CircNum}[1]{\ooalign{\hfil\raise .00ex\hbox{\scriptsize #1}\hfil\crcr\mathhexbox20D}}

\newcommand{\X}{B}
\newcommand{\Mon}{\mathrm{Mon}}

\newcommand{\map}{\mathrm{map}}

\newcommand{\fr}{\mathrm{fr}}
\newcommand{\Spin}{\mathrm{Spin}}
\newcommand{\Pin}{\mathrm{Pin}}

\let\emptyset\varnothing

\usepackage[all]{xy}
\SelectTips{cm}{10}
\CompileMatrices
\title[Framed, $r$-Spin and Pin moduli spaces and mapping class groups]{Homology of the moduli spaces and mapping class groups of framed, $r$-Spin and Pin surfaces}
\author{Oscar Randal-Williams}
\thanks{The author has been supported by an EPSRC Studentship, DTA grant number EP/P502667/1, ERC Advanced Grant No.\ 228082, and the Danish National Research Foundation through the Centre for Symmetry and Deformation.}
\email{o.randal-williams@dpmms.cam.ac.uk}
\address{Centre for Mathematical Sciences\\
Wilberforce Road\\
Cambridge CB3 0WB\\
UK}
\date{\today}

\begin{document}

\begin{abstract}
We give definitions of moduli spaces of framed, $r$-Spin and $\Pin^\pm$ surfaces. We apply earlier work of the author to show that each of these moduli spaces exhibits homological stability, and we identify the stable integral homology with that of certain infinite loop spaces in each case. We further show that these moduli spaces each have path components which are Eilenberg--MacLane spaces for the framed, $r$-Spin and $\Pin^\pm$ mapping class groups respectively, and hence we also identify the stable group homology of these groups. 

In particular: the stable framed mapping class group has trivial rational homology, and its abelianisation is $\bZ/24$; the rational homology of the stable $\Pin^\pm$ mapping class groups coincides with that of the non-orientable mapping class group, and their abelianisations are $\bZ/2$ for $\Pin^+$ and $(\bZ/2)^3$ for $\Pin^-$.
\end{abstract}
\maketitle

\input{chap1.tex}

\input{chap2.tex}
\input{chap3.tex}

\input{chap4.tex}

\appendix
\input{app1.tex}

\bibliographystyle{plain}
\bibliography{ExoticMCG}  

\end{document}

%% file: chap1.tex
\section{Introduction and statement of results}

Recent advances in the theory of moduli spaces of complex curves \cite{MW, H} concern the stable topology of these spaces, that is, the topology of moduli spaces of curves of high genus. This theory has been built upon by several workers to---amongst other things---deal with tangential structures other than orientations. Thus much is known about the homology of moduli spaces of unoriented surfaces \cite{Wahl}, moduli spaces of Spin surfaces \cite{HSpin, Bauer, galatius-2005}, and moduli spaces of oriented surfaces with maps to a simply-connected background space \cite{CM, CM2}.

The author has recently given \cite{R-WResolution} a general theory of homological stability for moduli spaces of surfaces with $\theta$-structure (which we define below). This recovers the above examples, but also allows one to effectively study many moduli spaces that are new to the literature. In this paper we study three examples of these: moduli spaces of framed surfaces, moduli spaces of $r$-Spin surfaces, and moduli spaces of $\Pin^\pm$ surfaces. 

\vspace{2ex}

Let us give a precise definition of the moduli spaces we have in mind. Write $F$ for a smooth surface, possibly with boundary. Let $\theta : \X \to BO(2)$ be a Serre fibration, and $\gamma_2 \to BO(2)$ be the universal bundle, so we obtain a 2-dimensional vector bundle $\theta^* \gamma_2 \to \X$. We let $\Bun(TF, \theta^*\gamma_2)$ denote the space of bundle maps $TF \to \theta^*\gamma_2$, i.e.\ fibrewise linear isomorphisms. Given a bundle map $\delta : TF|_{\partial F} \to \theta^*\gamma_2$, we let $\Bun_\partial(TF, \theta^*\gamma_2;\delta)$ denote the space of bundle maps $TF \to \theta^*\gamma_2$ that restrict to $\delta$ on the boundary. (If the surface $F$ is closed, $\partial F = \emptyset$ and there can be no $\delta$.) Let $\Diff_\partial(F)$ denote the group of diffeomorphisms of $F$ which restrict to the identity diffeomorphism on some neighbourhood of the boundary, and equip it with the $C^\infty$ topology.

\begin{defn}
The \textit{moduli space of surfaces with $\theta$-structure of topological type $F$ and boundary condition $\delta$} is the homotopy quotient
$$\mathcal{M}^\theta(F;\delta) := (\Bun_\partial(TF, \theta^*\gamma_2;\delta) \times E\Diff_\partial(F)) / \Diff_\partial(F),$$
where the group acts diagonally.
\end{defn}

If we define $\mathcal{E}^\theta(F;\delta) := (\Bun_\partial(TF, \theta^*\gamma_2;\delta) \times F \times E\Diff_\partial(F))/ {\Diff_\partial(F)}$, where the group acts triagonally, then the projection map
$$\mathcal{E}^\theta(F;\delta) \lra \mathcal{M}^\theta(F;\delta)$$
is a smooth $F$-bundle equipped with a bundle map $T^v \mathcal{E}^\theta(F;\delta) \to \theta^*\gamma_2$ from the vertical tangent bundle satisfying appropriate boundary conditions, and is universal with this property. This construction is discussed in more depth in \cite{GMTW, GR-W, R-WResolution}.

\begin{defn}
Given a $\theta$-structure $\delta$ on $\partial F$, and a point $\xi \in \mathcal{M}^\theta(F;\delta)$, we define the \textit{$\theta$ mapping class group} (at $\xi$) to be the fundamental group
$$\Gamma^\theta(F;\xi) := \pi_1(\mathcal{M}^\theta(F;\delta), \xi).$$
\end{defn}
If we do not wish to introduce notation for a boundary condition, we may write $\mathcal{M}^\theta(F)$ to denote $\mathcal{M}^\theta(F;\delta)$ with an unspecified $\delta$.

\vspace{2ex}

These moduli spaces have certain \textit{stabilisation maps} between them. Let $F$ be a surface with boundary condition $\delta : TF|_{\partial F} \to \theta^*\gamma_2$ and $F'$ be a surface with $\theta$-structure $\ell_{F'} : TF' \to \theta^*\gamma_2$. If we are given collections of boundary components $\partial_0 F \subset \partial F$ and $\partial_0 F' \subset \partial F'$, and an identification $\psi : \partial_0 F \cong \partial_0 F'$ such that $\psi^*(\ell_{F'}\vert_{\partial_0 F'}) = \delta\vert_{\partial_0 F}$, then there is a map
$$\mathcal{M}^\theta(F;\delta) \lra \mathcal{M}^\theta(F \cup_\psi F';\delta')$$
obtained by gluing $F'$ to $F$ along the identified boundaries. 

If we write $\Sigma_{g, b}$ for the orientable surface of genus $g$ with $b$ boundary components, these stabilisation maps between orientable surfaces are generated by certain elementary stabilisation maps
\begin{eqnarray*}
\alpha(g): \mathcal{M}^\theta(\Sigma_{g, b};\delta) \lra \mathcal{M}^\theta(\Sigma_{g+1, b-1};\delta')\\
\beta(g): \mathcal{M}^\theta(\Sigma_{g, b};\delta) \lra \mathcal{M}^\theta(\Sigma_{g, b+1};\delta')\\
\gamma(g): \mathcal{M}^\theta(\Sigma_{g, b};\delta) \lra \mathcal{M}^\theta(\Sigma_{g, b-1};\delta')
\end{eqnarray*}
given by gluing on a pair of pants along the legs, a pair of pants along the waist, and a disc, respectively. If we write $S_{n, b}$ for the non-orientable surface of genus $n$ with $b$ boundary components, in addition to the analogues of the above maps there are also stabilisation maps
\begin{eqnarray*}
\mu(n): \mathcal{M}^\theta(S_{n, b};\delta) \lra \mathcal{M}^\theta(S_{n+1, b};\delta')
\end{eqnarray*}
given by gluing on a projective plane with two discs removed.

\vspace{2ex}

Theorems about the stable topology of the moduli spaces $\mathcal{M}^\theta(F)$ typically take the form of stating that a comparison map to a certain infinite loop space is a homology equivalence in some range of degrees.

\begin{defn}\label{defn:MTtheta}
The \textit{Madsen--Tillmann spectrum} of $\theta$, denoted $\MTtheta$, is the Thom spectrum of the virtual bundle $-\theta^* \gamma_2 \to \X$. We denote by $\Omega^\infty \MTtheta$ the associated infinite loop space.
\end{defn}
For closed surfaces $F$ there is a natural comparison map
$$\alpha_F : \mathcal{M}^\theta(F) \lra \Omega^\infty \MTtheta$$
defined using Pontrjagin--Thom theory, and many characteristic classes of surface bundles with $\theta$-structure exist universally (i.e.\ independently of the topological type of $F$) on $\Omega^\infty \MTtheta$.

\subsection{Moduli spaces of framed surfaces} As our first example, let us take the tangential structure $\theta : EO(2) \to BO(2)$ corresponding to framings, and write $\mathcal{M}^{\fr}(\Sigma_{g, b};\delta)$ for the moduli space of framed surfaces with underlying surface $\Sigma_{g, b}$ and fixed framing $\delta$ along the boundary. Write $\Gamma^{\fr}(\Sigma_{g, b};\xi)$ for the framed mapping class group. We will show that the path component of $\mathcal{M}^{\fr}(\Sigma_{g, b};\delta)$ containing $\xi$ is homotopy equivalent to $B \Gamma^{\fr}(\Sigma_{g, b};\xi)$, so homological questions about the framed mapping class group are equivalent to homological questions about the moduli space of framed surfaces.

Our main theorem concerning the moduli spaces of framed surfaces is that they exhibit homological stability: the homology groups $H_*(\mathcal{M}^{\fr}(\Sigma_{g, b};\delta);\bZ)$ are independent of $g$, $b$ and $\delta$ as long as $6* \leq 2g-8$. Furthermore, the stable homology coincides with the homology of the space $\Omega^\infty \MT{EO}{2} = \Omega^\infty \mathbf{S}^{-2} \simeq \Omega^2 Q_0S^0$, the double loop space of the free infinite loop space on a point. In particular, it follows that the abelianisation of the group $\Gamma^{\fr}(\Sigma_{g, b};\xi)$ is $\bZ/24$ as long as $g \geq 7$, and that the rational group homology of $\Gamma^{\fr}(\Sigma_{g, b};\xi)$ is trivial in the stable range.

\subsection{Moduli spaces of $r$-Spin surfaces}
Recall that $\Spin^r(2) = U(1)$, but its standard 1-dimensional (complex) representation is the $r$-th tensor power of the standard representation of $U(1)$.

Fixing an $r$, we take the tangential structure $\theta : B\Spin^r(2) \to BO(2)$, and write $\mathcal{M}^{\Spin^r}(\Sigma_{g, b};\delta)$ for the moduli space of $r$-Spin surfaces with underlying surface $\Sigma_{g, b}$ and fixed $r$-Spin structure $\delta$ along the boundary. Write $\Gamma^{\Spin^r}(\Sigma_{g, b};\xi)$ for the $r$-Spin mapping class group. We will show that the component of $\mathcal{M}^{\Spin^r}(\Sigma_{g, b};\delta)$ containing $\xi$ is homotopy equivalent to $B \Gamma^{\Spin^r}(\Sigma_{g, b};\xi)$, so homological questions about the $r$-Spin mapping class group are equivalent to homological questions about the moduli space of $r$-Spin surfaces.

Our main theorem concerning the moduli spaces of $r$-Spin surfaces is that they exhibit homological stability: the homology groups $H_*(\mathcal{M}^{\Spin^r}(\Sigma_{g, b};\delta);\bZ)$ are independent of $g$, $b$ and $\delta$ in degrees $6* \leq 2g-8$. For $r=2$, the case of ordinary Spin Riemann surfaces, we can do better: the homology groups $H_*(\mathcal{M}^{\Spin^2}(\Sigma_{g, b};\delta);\bZ)$ are independent of $g$, $b$ and $\delta$ in degrees $5* \leq 2g-7$.

In the case $r=2$, our 2-Spin mapping class groups coincide with the extended Spin mapping class groups of Masbaum \cite{Masbaum}. Galatius has shown \cite{galatius-2005} that the stable homology of these groups coincides with the homology of the infinite loop space $\Omega^\infty_{0} \MT{Spin}{2}$. We show that a similar description of the stable homology is possible for all $r$, and give computational applications of this result in \cite{RWPicardrSpin}.

\subsection{Moduli spaces of {$\Pin^\pm$} surfaces} Recall that there are two generalisations of the covering group $\Spin(2) \to SO(2)$ to a covering group of $O(2)$, called $\Pin^+(2)$ and $\Pin^-(2)$, whose definitions we recall in \S\ref{sec:PinSurfaces}. These give Serre fibrations $\theta: B\Pin^\pm(2) \to BO(2)$ with corresponding moduli spaces $\mathcal{M}^{\Pin^\pm}(S_{n, b};\delta)$ and mapping class groups $\Gamma^{\Pin^\pm}(S_{n, b};\xi)$. We will show that the component of $\mathcal{M}^{\Pin^\pm}(S_{n, b};\delta)$ containing a point $\xi$ is homotopy equivalent to $B\Gamma^{\Pin^\pm}(S_{n, b};\xi)$, so homological questions about the $\Pin^\pm$ mapping class group are equivalent to homological questions about the $\Pin^\pm$ moduli space.

Our main theorems concerning these moduli spaces is that they exhibit homological stability (for non-orientable surfaces), and we give precise stability ranges in \S\ref{sec:HomStabPin}. Furthermore, the stable homology coincides with that of the infinite loop spaces $\Omega^\infty \MT{Pin^+}{2}$ and $\Omega^\infty \MT{Pin^-}{2}$ respectively. In particular, we are able to calculate that the abelianisation of the group $\Gamma^{\Pin^+}(S_{n, b};\delta)$ is $\bZ/2$ as long as $n \geq 10$, and that of $\Gamma^{\Pin^-}(S_{n, b};\delta)$ is $(\bZ/2)^3$ as long as $n \geq 13$. In \S\ref{sec:StableHomologyPin} we also study the divisibility of certain characteristic classes $\zeta_i$ which were defined by Wahl in the integral cohomology of the moduli spaces of unoriented surfaces, when they are pulled back to the moduli spaces $\mathcal{M}^{\Pin^\pm}(F)$.

\subsection{Guide}
We prove the homological stability theorems for these tangential structures by applying the general stability theorems of \cite{R-WResolution}. In order to verify the hypotheses of these theorems it is necessary to obtain a good understanding of the sets of path components $\pi_0(\mathcal{M}^\theta(F;\delta))$ and the effect of stabilisation maps between these sets. The bulk of the paper is dedicated to this problem for the tangential structures in question, and can be understood without reference to \cite{R-WResolution}. In the proofs of the homological stability theorems we refer to concepts defined in \cite[\S 7]{R-WResolution}, and we will not give these definitions again.

%% file: chap2.tex
\section{Moduli spaces of framed and $r$-Spin surfaces}

In the introduction we explained how a Serre fibration $\theta : \X \to BO(2)$ produces a moduli space of surfaces with $\theta$-structure $\mathcal{M}^\theta(F;\delta)$ for each surface $F$ and boundary condition $\delta$. In particular, if $\rho : G \to O(2)$ is a group homomorphism we may form a Serre fibration $\theta_\rho : BG \to BO(2)$ by $(EO(2) \times EG)/G \to EO(2)/O(2)$, where $G$ acts diagonally on $EO(2) \times EG$, via $\rho$ on the first factor.

\begin{defn}
To define $\mathcal{M}^{\Spin^r}(F;\delta)$ with $r \geq 2$, the \emph{moduli space of $r$-Spin surfaces} of topological type $F$ and boundary condition $\delta$, we perform the above construction with the homomorphism $U(1) \overset{(-)^r}\to U(1) = SO(2) \to O(2)$. We call the map $\theta^r : B\Spin^r(2) \to BO(2)$ and the bundle it classifies $\gamma_2^r$.

To define $\mathcal{M}^{\fr}(F;\delta)$, the \emph{moduli space of framed surfaces} of topological type $F$ and boundary condition $\delta$, we perform the above construction with the homomorphism $\{e\} \to O(2)$. We call the map $\theta^0 : B\Spin^0(2) \to BO(2)$ and the bundle it classifies $\gamma_2^0$.
\end{defn}

The maps defining both of these tangential structures naturally factor as
$$\theta^r : B\Spin^r(2) \overset{\theta^{r,+}}\lra BSO(2) \overset{\theta^+}\lra BO(2),$$
where the second map is the double cover classifying $\gamma_2^+$, the universal oriented rank 2 vector bundle. If $\ell : TF \to \gamma_2^r$ is a $\theta^r$-structure on $F$, we denote by $\ell^+ : TF \to \gamma_2^+$ the underlying orientation obtained from the map $\gamma_2^r \to \gamma_2^+$. This defines a map
$$\Bun_\partial(TF, \gamma_2^r;\delta) \lra \Bun_\partial(TF, \gamma_2^+;\delta^+)$$
whose fibre over a $\theta^+$-structure $\ell^+$ is homotopy equivalent to the space of lifts in the diagram
\begin{equation*}
\xymatrix{
& B\bZ/r \ar[d]\\
\partial F \ar[r]^-{\delta}\ar@{^(->}[d]& B\Spin^r(2)\ar[d]^-{\theta^{r,+}}\\
F \ar[r]^-{\ell^+} \ar@{-->}[ru] & BSO(2)
}
\end{equation*}
which is either empty (if $\ell^+$ does not admit a refinement to a $\theta^r$-structure) or else is \emph{non-canonically} homotopy equivalent to $\map_*(F / \partial F, B\bZ/r)$, where if $\partial F = \emptyset$ we interpret $F / \partial F$ as $F$ with a disjoint basepoint adjoined. More precisely, the group $\map_*(F / \partial F, B\bZ/r)$ acts on the space of lifts, and the orbit map given by any choice of lift is a homotopy equivalence.

Hence there are principal fibrations
\begin{eqnarray}
\map_*(F/\partial F, SO(2)) \lra \mathcal{M}^{\fr}(F;\delta) \lra \mathcal{M}^{+}(F;\delta^+) \label{eq:FibSeqFr}\\
\map_*(F/\partial F, B\bZ/r) \lra \mathcal{M}^{\Spin^r}(F;\delta) \lra \mathcal{M}^{+}(F;\delta^+)\label{eq:FibSeqrSpin}
\end{eqnarray}
whenever the total space is non-empty. Note that if $F$ has no boundary, $\mathcal{M}^{\mathrm{fr}}(F)$ is non-empty if and only if $F$ is diffeomorphic to a torus. Thus when discussing framed surfaces we will always suppose that they have boundary. Furthermore, from now on we will assume that all surfaces are orientable.

\subsection{Naturality properties}

If $r'$ divides $r$, there is a map of moduli spaces
$$\mathcal{M}^{\Spin^r}(F;\delta_r) \lra \mathcal{M}^{\Spin^{r'}}(F;\delta_{r'}),$$
where $\delta_{r'}$ is the induced $r'$-Spin structure from the $r$-Spin structure $\delta_r$. Furthermore, there are maps to all of these moduli spaces from $\mathcal{M}^{\mathrm{fr}}(F;\delta)$.

\subsection{Mapping class groups}

In the introduction we defined the $\theta$ mapping class group of a $\theta$-surface $\xi \in \mathcal{M}^\theta(F;\delta)$ to be the fundamental group based at this point.

In the case of framings---as $\partial F$ is assumed to be non-empty---in the fibration (\ref{eq:FibSeqFr}) we have an identification of the fibre through $\xi$ with the space $\map_*(F/\partial F, SO(2))$, which is homotopy-discrete and so we have an exact sequence of groups and pointed sets
$$0 \lra \pi_1(\mathcal{M}^{\fr}(F;\delta), \xi) \lra \Gamma^+(F) \overset{\varphi} \lra H^1(F, \partial F; \bZ) \lra \pi_0(\mathcal{M}^{\mathrm{fr}}(F;\delta)) \lra *.$$ 
Here $\Gamma^+(F)$ denotes the usual mapping class group of the oriented surface $F$ relative to its boundary, and the map $\varphi$ coincides with the crossed homomorphism obtained by Trapp \cite{Trapp} which gives an extended symplectic representation of the oriented mapping class group, though we will not pursue this connection. The above sequence identifies the framed mapping class group $\Gamma^{\fr}(F;\xi)$ as the subgroup of the oriented mapping class group $\Gamma^{+}(F)$ consisting of those (isotopy classes of) diffeomorphisms which fix the isotopy class of framings $[\xi]$. Furthermore, it implies that $\mathcal{M}^{\fr}(F;\delta)$ is a disjoint union of $K(\pi, 1)$'s, i.e.\ has the homotopy type of a groupoid.

In the case of $r$-Spin structures, if we suppose that $\partial F$ is non-empty we likewise obtain an exact sequence of groups and pointed sets
$$0 \lra \pi_1(\mathcal{M}^{\Spin^r}(F;\delta), \xi) \lra \Gamma^+(F) \overset{\varphi} \lra H^1(F, \partial F; \bZ/r) \lra \cdots.$$
This identifies the $r$-Spin mapping class group $\Gamma^{\Spin^r}(F;\xi)$ as the subgroup of the oriented mapping class group $\Gamma^{+}(F)$ consisting of those (isotopy classes of) diffeomorphisms which fix the isotopy class of $r$-Spin structures $[\xi]$. 

If $F$ does not have boundary, then we instead obtain a sequence
$$0 \lra \bZ/r \lra \pi_1(\mathcal{M}^{\Spin^r}(F), \xi) \lra \Gamma^+(F) \overset{\varphi} \lra H^1(F; \bZ/r) \lra \cdots$$
which identifies the $r$-Spin mapping class group with an extension by $\bZ/r$ of the subgroup of $\Gamma^+(F)$ of mapping classes that preserve a $r$-Spin structure up to isomorphism. In either case, $\mathcal{M}^{\Spin^r}(F;\delta)$ is a disjoint union of $K(\pi, 1)$'s, i.e.\ has the homotopy type of a groupoid.

\subsection{The set of $\theta^r$-structures}
In order to apply the results of \cite{R-WResolution}, we must calculate the set of path components $\pi_0(\mathcal{M}^{\theta^r}(F;\delta))$, that is, calculate the set of isotopy classes of $\theta^r$-structures a surface admits up to diffeomorphism of the underlying surface. This coincides with the quotient set of $\pi_0\Bun_\partial(TF, \gamma_2^{\fr};\delta)$ or $\pi_0\Bun_\partial(TF, \gamma_2^{r};\delta)$ by the action of the unoriented mapping class group $\Gamma(F)$. We would rather work with the ordinary orientation-preserving mapping class group $\Gamma^+(F)$, so we fix an orientation of $F$ and define the subspace
$$\Bun_\partial^+(TF, \gamma_2^{r};\delta) \subset \Bun_\partial(TF, \gamma_2^{r};\delta)$$
of those bundle maps $\ell : TF \to \gamma_2^r$ such that $\ell^+$ is the given orientation of $F$, and define
$$\theta^r(F;\delta) := \pi_0(\Bun^+_\partial(TF, \gamma_2^{r};\delta)).%\begin{cases}
%\pi_0\Bun_\partial(TF, \gamma_2^{\fr};\delta) & \text{if $r=0$}\\
%\pi_0\Bun_\partial(TF, \gamma_2^{r};\delta) & \text{if $r \geq 2$.}
%\end{cases}
$$
This set has an action of $\Gamma^+(F)$, and as described earlier, is either empty or has an action of $H^1(F, \partial F ; \bZ/r)$ which is free and transitive, i.e.\ is a torsor.

\begin{defn}
Choose once and for all a $\theta^r$-structure on $\bR^2$, that is, a linear map $\bR^2 \to \gamma_2^r$. If $V \to B$ is a framed rank two vector bundle, the \emph{standard $\theta^r$-structure} is the fibrewise linear isomorphism $V \to \bR^2 \to \gamma_2^r$.
\end{defn}

On the surface $\Sigma_{g, b+1}$ having $\theta^r$-structure $\xi$ with boundary condition $\delta$, choose a marked point on each boundary component, and equip each marked point with the framing coming from the orientation and the inwards pointing normal vector. After perhaps changing $\delta$ to an isomorphic boundary condition, we may suppose that the $\theta^r$-structure is standard at each marked point.

Let $\{a_i, b_i, t_i, \partial_i, r_i\}$ be the collection of simple closed curves and simple arcs in $\Sigma_{g,b+1}$ as shown in Figure \ref{fig:OrientableCycles}, where the endpoints of the arcs lie at the standard marked points on each boundary component.
\begin{figure}[h]
\begin{center}
\includegraphics[bb=0 0 351 203]{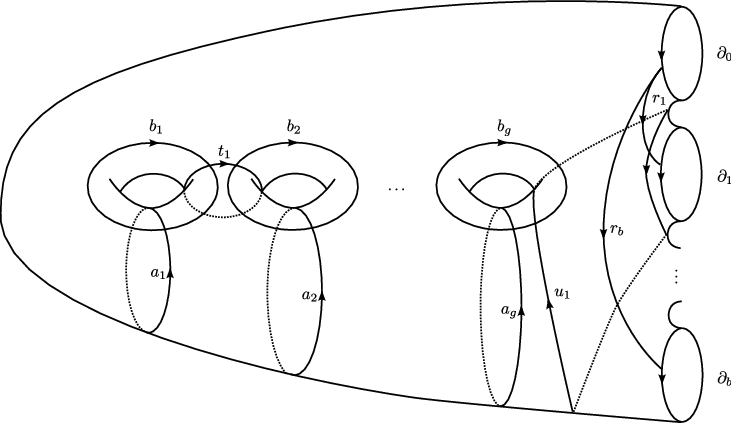}
\end{center}
\caption{Curves and arcs on $\Sigma_{g, b+1}$.}
\label{fig:OrientableCycles}
\end{figure}
Note that $a_i \cap b_j = \delta_{ij}$ and $a_i \cap a_j = b_i \cap b_j =0$. These curves have canonical lifts $\tilde{a}_i$, $\tilde{b}_i$, $\tilde{t}_i$, $\tilde{\partial}_i$, $\tilde{r}_i$ to curves or arcs on $\bS\Sigma_{g,b+1}$, given by assigning them their unit forward tangent vector at each point. We also have the loop $z$ in $\bS\Sigma_{g,b+1}$ given by a single fibre, with orientation given by that of $\Sigma_{g,b+1}$. We remark that these lifts are \textit{not} homology invariant: although $t_1$ is homologous to $a_2 - a_1$, $\tilde{t}_1$ is homologous to $\tilde{a}_2-\tilde{a}_1+z$. More generally, when lifting homologous elements from $\Sigma_{g,b+1}$ to $\bS \Sigma_{g,b+1}$, there is a correction term given by the Euler characteristic of a homology chain times $z$.

\begin{defn}
Given a $\theta^r$-structure $\xi$ on $\Sigma_{g, b+1}$ we define a $\bZ/r$-valued function $q_\xi$ on the set of simple closed curves (or simple arcs between the standard marked points on each boundary component) by assigning to each curve or arc $x$ the value $q_\xi(x)$ determined as follows: $TF \vert_x$ has a $\theta^r$-structure, but the forward vector field along $x$ splits off a trivial 1-dimensional sub-bundle, so reduces the structure group of this bundle to $\bZ/r$. If $x$ is a simple closed curve, the monodromy gives an element of $\bZ/r$. If $x$ is an arc, the framing of $T\Sigma_{g,b+1}\vert_x$ given by the forward vector (and the orientation) agrees with the standard framing at the start of the arc, but not at the end: here it differs by a half rotation. However, the orientation of the surface gives a canonical choice of half rotation, which makes the $\theta^r$-structure on the arc be standard near its ends, and hence gives an element of $\bZ/r$. In both cases we denote the element of $\bZ/r$ obtained by $\Mon(x)$.

We define $q_\xi(x)$ to be $\Mon(x) - 1 \in \bZ/r$.
\end{defn}

The reader may be at a loss as to why we subtract 1 to what is already a perfectly good invariant: it is a normalisation, and ensures that if a simple closed curve $x$ bounds a disc, then $q_\xi(x)=0$. We may now define a function
$$p: \theta^r(\Sigma_{g,b+1};\delta) \lra (\bZ/r)^{2g+b}$$
given by $p(\xi) := (q_\xi(a_1), q_\xi(b_1), \ldots, q_\xi(a_g), q_\xi(b_g), q_\xi(r_1), \ldots, q_\xi(r_b))$.

\begin{prop}\label{prop:Coordinates}
If $\theta^r(\Sigma_{g, b+1};\delta)$ is non-empty, the map $p$ is a bijection.
\end{prop}
\begin{proof}
Note that once a $\theta^r$-structure is determined over the curves $a_i$, $b_i$ and $r_j$, it remains to give a $\theta^r$-structure on a disc satisfying a certain boundary condition (which up to isomorphism depends only on $\delta$, and not on the values of $\xi$ on the curves). This is possible (and if so, in a unique way) if and only if $\theta^r(\Sigma_{g, b+1};\delta)$ is non-empty, by obstruction theory for the map $\theta^r$.
\end{proof}

Using this proposition, we may study the action of the mapping class group $\Gamma^+(\Sigma_{g,b+1})$ on the set $\theta^r(\Sigma_{g, b+1};\delta)$ by studying its action in the ``coordinate system'' given by $p$.

\begin{rem}
We should mention to what extent the function $p$ is unique. We have only defined it for boundary conditions with a choice of standard marked point on each boundary, and an ordering of the boundaries. Varying either of these choices will vary the function.
\end{rem}

If $x$ is a simple closed curve or simple arc on $F$, then we denote by $\tilde{x}$ its canonical lift to $\bS F$ given by the forward tangent vector.

\begin{lem}\label{lem:ThetaStructFundClass}
A $\theta^r$-structure $\xi$ on $F$ gives an element $[\xi] \in H^1(\bS F; \bZ/r)$, such that the function $q_\xi$ on a simple closed curve or arc with standard ends $x$ is given by 
$$q_\xi(x) =  [\xi](\tilde{x}) -1 \in \bZ/r.$$
\end{lem}
\begin{proof}
By a Serre spectral sequence calculation, there is a unique class in $H^1(\bS \gamma_2^r ; \bZ/r)$ which restricts to $1 \in H^1(S^1;\bZ/r)$, as the Euler class of $\gamma_2^r$ is divisible by $r$. This gives a map $\bS \gamma_2^r \to B\bZ/r$ which we see is a homotopy equivalence. A $\theta^r$-structure on $F$ gives a map $\bS F \to \bS \gamma_2^r$ taking the standard part of the boundary to the basepoint and we define $[\xi]$ to be the pullback of the tautological class via this map. The claimed property is now immediate.
\end{proof}

Let us write $\tau_a$ for the (forward) Dehn twist around a simple closed curve $a$, and recall that the action of a twist on a homology class $x$ is given by the formula $\tau_a(x) = x + \langle a,x \rangle \cdot a$, where $\langle\,\, ,\, \rangle$ denotes the intersection product. A Dehn twist is a diffeomorphism, so in particular gives a map of the sphere bundle: we use $\tau_a$ to denote this map also.

\begin{lem}
We have $\tau_a(\tilde{x}) = \widetilde{\tau_a(x)}$, which is homologous to $\tilde{x} + \langle a,x \rangle \cdot \tilde{a}$. Hence we have the formula
$$q_{\tau_a^*\xi}(x) = \tau_a^*(q_\xi)(x) = q_\xi(x) + \langle a,x\rangle (q_\xi(a)+1)$$
for the action of Dehn twists on $\theta^r$-structures.
\end{lem}
\begin{proof}
The first part may be seen by constructing an explicit homology in $\bS F$ between $\widetilde{\tau_a(x)}$ and $\tilde{x} + \langle a, x \rangle \cdot \tilde{a}$ concentrated near the intersection points of $a$ and $x$, when we represent them by transverse smooth 1-manifolds. The second part now follows as by Lemma \ref{lem:ThetaStructFundClass} the function $q_\xi$ is given by evaluating against a cohomology class on $\bS F$---which is linear---and then subtracting 1.
\end{proof}

It is also useful to observe that if $a$ and $b$ are disjoint simple closed curves and $a \# b$ is the simple closed curve obtained by forming the oriented connected sum of $a$ and $b$, then we obtain the equation on homology classes
$$[\widetilde{a \# b}] = [\tilde{a}] + [\tilde{b}] + [z]
 \in H_1(\bS F ; \bZ)$$
or equivalently for any $\theta^r$-structure $\xi$,
$$q_\xi(a \# b) = q_\xi(a) + q_\xi(b).$$
The same holds when one of $a$ and $b$ is a simple arc. More generally, if $x$ is a simple closed curve whose homology class may be written as $[x] = [a]+[b]$ for simple closed curves $a$, $b$, then
\begin{equation}\label{eq:QuadraticProperty}
q_\xi(x) = q_\xi(a) + q_\xi(b) + \langle a, b \rangle.
\end{equation}
This follows from the methods of Johnson \cite{Johnson} (in particular his Theorem 1B).

\subsection{Diffeomorphism classes of $\theta^r$-structures}\label{sec:DiffClassesThetaRStruct}

Once we have the bijection $p: \theta^r(\Sigma_{g,b+1};\delta) \lra (\bZ/r)^{2g+b}$, there is a surjective function $A : \theta^2(\Sigma_{g, b+1};\delta) \to \bZ/2$ given in terms of this bijection by the formula
\begin{equation}\label{eq:ArfInvariant}
\xi \mapsto \sum_{i=1}^g q_\xi(a_i) \cdot q_\xi(b_i) + \sum_{j=1}^b q_\xi(r_j) \cdot \delta_j,
\end{equation}
where $\delta_j := q_\xi(\partial_j)$. For surfaces with zero or one boundaries, the set of 2-Spin structures may be identified with the set of quadratic refinements of the intersection form \cite{Johnson} via $\xi \mapsto q_\xi$. In that case, the second term of this formula vanishes and $A$ is simply the Arf invariant of the quadratic form. Hence we call it the \emph{generalised Arf invariant} for 2-Spin surfaces with boundary. We remind the reader that the bijection $p$ was not canonical (it depended on a choice of trivialised marked point on each boundary and an ordering of the boundaries), and hence $A$ considered as a function on $\theta^2(\Sigma_{g, b+1};\delta)$ is not canonical either.

\begin{prop}\label{prop:ArfInvariantIsInvariant}
The function $A$ is $\Gamma^+(\Sigma_{g, b+1})$-invariant.
\end{prop}
\begin{proof}
Let $x$ be a simple closed curve represented in homology by $\sum_{i=1}^g X_i a_i + Y_i b_i + \sum_{j=0}^b \lambda_j \partial_j$. Then a 2-Spin structure $\xi = (A_1, B_1, \ldots, A_g, B_g, R_1, \ldots, R_b)$ evaluated on $x$ gives
$$q_\xi(x) = q_\xi \left(\sum_{i=1}^g X_i a_i + Y_i b_i + \sum_{j=0}^b \lambda_j \partial_j \right) = \sum_{i=1}^g X_i A_i + Y_i B_i + X_i Y_i + \sum_{j=0}^b \lambda_i \delta_i$$
using (\ref{eq:QuadraticProperty}). The Dehn twist around this curve satisfies
\begin{eqnarray*}
q_{\tau_x^*\xi}(a_i) = A_i + Y_i(q_\xi(x)+1)\\
q_{\tau_x^*\xi}(b_i) = B_i + X_i(q_\xi(x)+1)\\
q_{\tau_x^*\xi}(r_j) = R_j + (\lambda_0 + \lambda_j)(q_\xi(x)+1)
\end{eqnarray*}
so if $q_\xi(x)=1$ then $q_{\tau_x^* \xi} = q_\xi$ and the invariant is trivially preserved. If $q_\xi(x)=0$ then
\begin{eqnarray*}
\sum_{i=1}^g q_{\tau_x^*\xi}(a_i) \cdot q_{\tau_x^*\xi}(b_i) & = & \sum_{i=1}^g A_i B_i + \left ( \sum_{i=1}^g A_i X_i + B_i Y_i + X_i Y_i \right )
\end{eqnarray*}
and
\begin{eqnarray*}
\sum_{j=1}^r q_{\tau_x^*\xi}(r_i) \cdot \delta_j & = & \sum_{j=0}^r R_j \delta_j + \left ( \sum_{j=1}^b (\lambda_0 + \lambda_j)\delta_j \right )\\
 & = & \sum_{j=0}^r R_j \delta_j + \left ( \sum_{j=0}^b  \lambda_j\delta_j \right )
\end{eqnarray*}
and $0 = q_\xi(x) = \sum_{j=0}^b  \lambda_j\delta_j + \sum_{i=1}^g A_i X_i + B_i Y_i + X_i Y_i$ so $A(\xi) = A(\tau_a^*\xi)$.
\end{proof}

Thus we have produced a $\Gamma^+(\Sigma_{g, b+1})$-invariant function $A : \theta^2(\Sigma_{g, b+1}) \to \bZ/2$. This induces $\Gamma^+(\Sigma_{g, b+1})$-invariant functions on all $\theta^{2n}(\Sigma_{g,b+1})$, $n \geq 0$, by composing with the natural map $\theta^{2n}(\Sigma_{g,b+1};\delta) \to \theta^{2}(\Sigma_{g,b+1};\delta_2)$. The main result of this section is the following theorem, which determines the number of orbits of $\Gamma^+(\Sigma_{g, b+1})$ acting on $\theta^r(\Sigma_{g,b+1};\delta)$.

\begin{thm}\label{thm:OrbitCount}
Let $g \geq 2$. If the set $\theta^r(\Sigma_{g,b+1};\delta) / \Gamma^+(\Sigma_{g,b+1})$ is not empty:
\begin{enumerate}[(i)]
\item It consists of a single element, if $r$ is odd.

\item It consists of two elements distinguished by the invariant $A$, if $r$ is even.
\end{enumerate}

Let $g \geq 1$. If the set $\theta^2(\Sigma_{g,b+1};\delta) / \Gamma^+(\Sigma_{g,b+1})$ is not empty it consists of two elements distinguished by the invariant $A$.
\end{thm}

In order to prove this theorem we require the following lemma which describes the action of the mapping class group $\Gamma^+(\Sigma_{g,b+1})$ on the set $\theta^r(\Sigma_{g,b+1};\delta)$, in the coordinate system given by Proposition \ref{prop:Coordinates}.

\begin{lem}Translating the action of $\Gamma^+(\Sigma_{g,1})$ on $\theta^r(\Sigma_{g, 1};\delta)$ to $(\bZ/r)^{2g}$, Dehn twists around the cycles $a_i$ and $b_i$ give
$$\tau_{a_i}^{\pm 1} \cdot (A_1, B_1,\ldots, A_g, B_g) = (A_1, B_1, \ldots, B_{i-1}, A_i, B_i \pm (A_i+1), A_{i+1}, \ldots, A_g, B_g),$$
$$\tau_{b_i}^{\pm 1} \cdot (A_1, B_1,\ldots, A_g, B_g) = (A_1, B_1, \ldots, B_{i-1}, A_i \mp (B_i+1), B_i , A_{i+1}, \ldots, A_g, B_g).$$
Dehn twists around the cycles $t_i$ give that $\tau_{t_i} \cdot (A_1, B_1,\ldots, A_g, B_g)$ is
$$(A_1, \ldots, A_i, B_i+A_i-A_{i+1}-1, A_{i+1}, B_{i+1}+A_{i+1}-A_i+1, A_{i+2}, \ldots, B_g).$$
The analogous formulae hold for more than one boundary.
\end{lem}
\begin{proof}
For the first part, we compute $\tau_{a_i}(q_\xi)(b_i) = q_\xi(\tau_{a_i}(b_i)) = q_\xi(b_i) + q_\xi(a_i)+1$, and so on. For the second part, note that $t_i$ is homologous to $a_{i+1}-a_i$, and that $\tilde{t}_i$ is homologous to $\tilde{a}_{i+1} - \tilde{a}_i + z$. Thus $q_\xi(t_i) = q_\xi(a_{i+1}) - q_\xi(a_i)$,
$$\tau_{t_i}(q_\xi)(b_i) = q_\xi(b_i) - (q_\xi(t_i)+1) = q_\xi(b_i) + q_\xi(a_i) - q_\xi(a_{i+1}) - 1$$
and
$$\tau_{t_i}(q_\xi)(b_{i+1}) = q_\xi(b_{i+1}) + (q_\xi(t_i)+1) = q_\xi(b_{i+1}) - q_\xi(a_i) + q_\xi(a_{i+1}) + 1$$
which establishes the required formula.
\end{proof}

\begin{proof}[Proof of Theorem \ref{thm:OrbitCount}]
Consider the element
$$(A_1, B_1, \dots, A_g, B_g, R_1, \dots, R_b) \in \theta^r(\Sigma_{g, b+1};\delta).$$
We will show how to reduce this to an element in standard form.

By iteratedly applying $\tau_{a_i}$ and $\tau_{b_i}$ we can reduce the pair $(A_i, B_i)$ to the form $(-1, N)$. It is easy to check that the number $\gcd(A_i+1, B_i+1) \in \bZ/r$ is invariant under $\tau_{a_i}$ and $\tau_{b_i}$, so that $N = \gcd(A_i+1, B_i+1)-1$; let us call this $G(A_i, B_i)$. It is not necessary for the proof, but to relate this to the function $A$ we remark that $G(A_i, B_i) \equiv A_i B_i \,\, \mathrm{mod} \,\, 2$. Hence
\begin{equation}\label{eq:1}
(A_1, B_1, A_2, B_2, \dots, A_g, B_g, R_1, \dots R_b)
\end{equation}
is equivalent to
\begin{equation}\label{eq:2}
(-1, G(A_1, B_1), -1, G(A_2, B_2), \dots, -1, G(A_g, B_g), R_1, \dots R_b).
\end{equation}
Applying $\tau_{t_i}$ sends
$$(-1, G(A_i, B_i), -1, G(A_{i+1}, B_{i+1}))\quad \text{to} \quad(-1, G(A_i, B_i)-1, -1, G(A_{i+1}, B_{i+1})+1),$$
so applying it iteratedly to (\ref{eq:2}) gives
\begin{equation}\label{eq:3}
\bigg(-1, 0, -1, 0, \dots, -1, \sum_{i=1}^gG(A_i, B_i), R_1, \dots R_b\bigg).
\end{equation}
Let us write $N:= \sum_{i=1}^gG(A_i, B_i)$ temporarily. We will now show how to reduce the coordinates $R_i$ to a more standard form. Define the simple closed curve $u_i$ to be the connected sum of $a_g$ and $\partial_i$, formed as indicated in Figure \ref{fig:OrientableCycles} to only intersect $b_g$, so $q_\xi(u_i) = A_g + \delta_i$. Twisting along $u_1$ gives
$$\tau_{u_1} \cdot (-1,0,-1,0, \ldots, -1, N, R_1, \ldots, R_r) = (-1,0,-1,0, \ldots, -1, N + \delta_1, R_1+\delta_1, \ldots, R_r),$$
and twisting backwards around $\partial_1$ gives $(-1,0,-1,0, \ldots, -1, N + \delta_1, R_1 - 1, \ldots, R_r)$. Repeating $R_1$ times gives $(-1,0,-1,0, \ldots, -1, N + R_1 \delta_1, 0, R_2, \ldots, R_r)$, and continuing in this way, by twisting around $u_2$, $u_3$, and so on, we can arrive at
\begin{equation}
(-1,0,-1,0, \ldots, -1, \sum_{i=1}^gG(A_i, B_i) + \sum R_i\delta_i, 0, \ldots, 0).
\end{equation}
This much holds for $g \geq 1$. In particular, if $r=2$ we have reduced to one of two possibilities, which are distinguished the invariant $A$, so we ahve proved the second part of the theorem.

If $g \geq 2$ we may apply the sequence
\begin{eqnarray}
\nonumber (-1,0, -1, N) \overset{\tau_{t_{g-1}}^2}\sim (-1,-2,-1, N+2) \overset{\tau_{b_{g-1}}}\sim (0,-2,-1,N+2)\\
\nonumber \overset{\tau_{a_{g-1}}^2}\sim (0, 0, -1, N+2) \overset{\tau_{b_{g-1}}}\sim (-1,0,-1,N+2)
\end{eqnarray}
to find that there are at \textit{most} two orbits. If $r$ is even the invariant $A$ shows there are at \emph{least} two orbits, and we are done. If $r$ is odd note that we have put everything in the form $(-1,0,\dots, -1, X, 0, \dots, 0)$ for $X \in \bZ/r$, but this is equivalent to $(-1,0,\dots, -1, X+2, 0, \dots, 0)$. Thus there is a single orbit in this case.
\end{proof}

\subsection{Gluing $\theta^r$-surfaces}\label{sec:GluingThetaRSurfaces}

In this section we will discuss how the invariant $A$ behaves with respect to gluing $\theta^r$-surfaces. In order to do so effectively, it is convenient to discuss \emph{connected cobordisms with $\theta^r$-structure}. This simply means that we have designated incoming and outgoing boundaries, the marked points on the outgoing boundaries are framed using the boundary orientation and the inwards normal vector, and the marked points on the incoming boundaries are framed using the boundary orientation and the outwards normal vector. Furthermore, we have an ordering of first the outgoing boundaries and then the incoming boundaries, with respect to which we compute the generalised Arf invariant. We call this data a \emph{prepared $\theta^r$-surface}.
%FIGURE

Given cobordisms with boundary condition $(\Sigma, \delta)$ and $(\Sigma', \delta')$, and an identification $\psi : {\partial}_{out} \Sigma \cong {\partial}_{in} \Sigma'$ between incoming and outgoing boundaries such that $\psi^*(\delta'\vert_{\partial_{in}\Sigma'}) = \delta\vert_{\partial_{out}\Sigma}$, we obtain a gluing map
$$G_\psi : \theta^r(\Sigma;\delta) \times \theta^r(\Sigma';\delta') \lra \theta^r(\Sigma \cup_\psi \Sigma';\delta \cup \delta')$$
which induces a map
$$G_\psi : \theta^r(\Sigma;\delta) / \Gamma^+(\Sigma) \times \theta^r(\Sigma';\delta') / \Gamma^+(\Sigma) \lra \theta^r(\Sigma \cup_\psi \Sigma';\delta \cup \delta') / \Gamma^+(\Sigma \cup_\psi \Sigma').$$
Under the identification of $\theta^r(\Sigma;\delta) / \Gamma^+(\Sigma)$ with $\pi_0(\mathcal{M}^{\theta^r}(\Sigma;\delta))$, this is nothing but the gluing map between these moduli spaces, at the level of $\pi_0$. 

We will determine the effect of this map on the invariant $A$. In order for this to be meaningful, we must declare how to impose the data of a prepared $\theta^r$-surface on $(\Sigma \cup_\psi \Sigma';\delta \cup \delta')$: we use the marked points of the two surfaces and the induced ordering on the unglued boundaries.

\begin{lem}\label{lem:AddivityOfA}
Let $r$ be even, and $\Sigma$ and $\Sigma'$ be connected prepared $\theta^r$-surfaces with a gluing datum $\psi$ as above. The function $A : \theta^2(\Sigma \cup_\psi \Sigma'; \delta \cup \delta') \to \bZ/2$ evaluated on the glued $\theta^r$-structure $G_\psi(\ell_\Sigma, \ell_{\Sigma'})$ is given by
\begin{equation*}
A(\ell_\Sigma) + A(\ell_{\Sigma'}) + {\delta}\vert_{\partial_{out} \Sigma}.
\end{equation*}
\end{lem}
\begin{proof}
This is immediate from the formula for $A$, by considering the new genus that may be formed in such a gluing. There are two important points:
\begin{enumerate}[(i)]
\item When gluing two arcs together at a single end, the total monodromy along them adds, and so the value of the invariant adds, \emph{but then one is added}.
\item When two arcs are joined to create a new simple closed curve, the value of the invariant along it is the sum of the values along the two curves. In light of the previous point this is counterintuitive, but has to do with the canonical ``straightening'' done to arcs ending at the outgoing boundary or starting at the incoming boundary: the two half turns do not cancel, they \emph{add}.\qedhere
\end{enumerate}
\end{proof}

In trying to understand this lemma, we remark that the reader should be aware that it is \emph{only stated for connected surfaces}. In fact one may prove that there is no function from Spin surfaces with boundary to $\bZ/2$ which is additive for gluing and disjoint union, and agrees with the Arf invariant on closed surfaces. 

\begin{cor}\label{cor:GluingBij}
Let $(\Sigma', \xi')$ be a surface with $\theta^r$-structure, $\Sigma$ have a boundary condition $\delta$, and $\psi$ be a gluing datum as above. Then the map
$$G_\psi(-,\xi') : \theta^r(\Sigma;\delta)/\Gamma(\Sigma) \lra \theta^r(\Sigma \cup_\psi \Sigma';\delta \cup \delta')/\Gamma(\Sigma \cup_\psi \Sigma')$$
is a bijection as long as either
\begin{enumerate}[(i)]
\item $\Sigma$ has genus at least 2, or
\item $r=2$ and $\Sigma$ has genus at least 1.
\end{enumerate}
\end{cor}
\begin{proof}
In the first case, by Theorem \ref{thm:OrbitCount} both sets are singletons if $r$ is odd, so there is nothing to prove, or consist of two elements distinguished by the invariant $A$ if $r$ is even. In this case, the formula of Lemma \ref{lem:AddivityOfA} shows that $A \circ G_\psi(-, \xi')$ is surjective, so $G_\psi(-, \xi')$ is too which proves the claim. The second case is the same, using the improved part of Theorem \ref{thm:OrbitCount} in the case $r=2$.
\end{proof}

The final tool we shall need is

\begin{lem}\label{lem:SurjGenus1}
In genus 1 the maps
$$\theta^r(\Sigma_{1,b};\delta) / \Gamma^+(\Sigma_{1,b}) \lra \theta^r(\Sigma_{1,b+1};\delta') / \Gamma^+(\Sigma_{1,b+1})$$
that glue on a new boundary are surjective. %In particular, if $r=2$ they are bijective.
\end{lem}
\begin{proof}
By the proof of Theorem \ref{thm:OrbitCount}, elements of $\theta^r(\Sigma_{1,b};\delta) / \Gamma^+(\Sigma_{1,b})$ may be represented by $(-1, X, 0, \ldots, 0)$. If we glue on the pair of pants represented by $(R_1, R_2)$ along its last boundary to the first boundary of $\Sigma$, we obtain the element
$$(-1, X, R_2+1, \ldots, R_2+1, R_1)$$
which may be reduced to $(-1, X + (R_2+1)\sum_{i=1}^{b-1} \delta_i + R_1 \delta'_1 , 0, \ldots, 0)$. By varying $X$ we see that may obtain every element of the form $(-1, Y, 0, \ldots, 0)$, as required.
%The claim for $r=2$ follows as in this case both sets consist of two elements by Theorem \ref{thm:OrbitCount}, so a surjection is a bijection. 
\end{proof}

\subsection{Homological stability for framed and $r$-Spin surfaces}

We will now describe how to use the above results to apply the main theorem of \cite{R-WResolution} to prove homological stability for framed and $r$-Spin surfaces. To do so we must necessarily use terminology introduced in \cite{R-WResolution}, but we restrict its use to the proof. Note that the statement for $r=0$ is that for framed surfaces.

\begin{thm}\label{thm:StabilityrSpin}
The moduli spaces of $\theta^r$-surfaces exhibit homological stability. More precisely,
\begin{enumerate}[(i)]
	\item Any $\alpha(g):\mathcal{M}^{\theta^r}(\Sigma_{g, b}) \to \mathcal{M}^{\theta^r}(\Sigma_{g+1, b-1})$ is a homology epimorphism in degrees $6* \leq 2g-2$ and a homology isomorphism in degrees $6* \leq 2g-8$.
	\item Any $\beta(g): \mathcal{M}^{\theta^r}(\Sigma_{g, b}) \to \mathcal{M}^{\theta^r}(\Sigma_{g, b+1})$ is a homology epimorphism in degrees $6* \leq 2g - 4$ and a homology isomorphism in degrees $6* \leq 2g-10$. If one of the created boundary conditions is trivial, it is a split homology monomorphism in all degrees.
	\item Any $\gamma(g): \mathcal{M}^{\theta^r}(\Sigma_{g,b}) \to \mathcal{M}^{\theta^r}(\Sigma_{g, b-1})$ is a homology isomorphism in degrees $6* \leq 2g - 4$. If $b \geq 2$ it is a split homology epimorphism in all degrees, and if $b=1$ it is a homology epimorphism in degrees $6* \leq 2g +2$.
\end{enumerate}
For $r=2$ we can do better. In this case,
\begin{enumerate}[(i)]
	\item Any $\alpha(g):\mathcal{M}^{\theta^2}(\Sigma_{g, b}) \to \mathcal{M}^{\theta^2}(\Sigma_{g+1, b-1})$ is a homology epimorphism in degrees $5* \leq 2g-2$ and a homology isomorphism in degrees $5* \leq 2g-7$.
	\item Any $\beta(g): \mathcal{M}^{\theta^2}(\Sigma_{g, b}) \to \mathcal{M}^{\theta^2}(\Sigma_{g, b+1})$ is a homology epimorphism in degrees $5* \leq 2g - 3$ and a homology isomorphism in degrees $5* \leq 2g-8$. If one of the created boundary conditions is trivial, it is a split homology monomorphism in all degrees.
	\item Any $\gamma(g): \mathcal{M}^{\theta^2}(\Sigma_{g,b}) \to \mathcal{M}^{\theta^2}(\Sigma_{g, b-1})$ is a homology isomorphism in degrees $5* \leq 2g - 3$. If $b \geq 2$ it is a split homology epimorphism in all degrees, and if $b=1$ it is a homology epimorphism in degrees $5* \leq 2g +2$.
\end{enumerate}
\end{thm}
\begin{proof}
In order to apply Theorems 8.1 and 12.4 of \cite{R-WResolution}, we must verify two conditions: that $\theta^r$-structures \emph{stabilise on $\pi_0$ at genus $h$}, and that they are \emph{$k$-trivial}. Theorem \ref{thm:OrbitCount} shows that for $g \geq 2$ the moduli spaces always have 1 (if $r$ is odd) or 2 (if $r$ is even) path components, and Lemma \ref{lem:AddivityOfA} shows that when $r$ is even then stabilisation maps induce bijections on path components. Furthermore, Lemma \ref{lem:SurjGenus1} shows that $\beta$ type stabilisation maps starting in genus $1$ induce surjections on path components, and it follows from Theorem \ref{thm:OrbitCount} and Lemma \ref{lem:AddivityOfA} that the same is true for $\alpha$ type stabilisation maps. Hence $\theta^r$-structures stabilise on $\pi_0$ at genus 2. Now, \cite[Proposition 7.6]{R-WResolution} shows that a formal consequence of stabilising on $\pi_0$ at genus $h$ is being $(2h+1)$-trivial, which implies that $\theta^r$-structures are 5-trivial. In general this is the best we will be able to do, but for $r=2$ we can do a little better and show $\theta^2$ is 4-trivial. Once we have done this, solving the recurrence relations of \cite[\S 7.5]{R-WResolution} with this data gives the stated stability ranges.

To establish 4-triviality of $\theta^2$ we first refer the reader to the definition of 4-triviality in \cite[\S 7.1]{R-WResolution}. Suppose that we are given data $((s, \xi_s), b, (\bar{b}_i, \xi_{\bar{b}_i})_{i=1}^4)$ as in \cite[\S 7.1]{R-WResolution} which is maximal in the sense of \cite[Definition 7.2]{R-WResolution}, and form the commutative square of embeddings of surfaces with $\theta^2$-structure
\begin{equation*}
\xymatrix{
F_{00} \ar@{^(->}[r]\ar@{^(->}[d]& F_{10} \ar@{^(->}[d]\\
F_{01} \ar@{^(->}[r]& F_{11}
}
\end{equation*}
which we must show may be trivialised. As in the proof of \cite[Proposition 7.6]{R-WResolution} we let $\widetilde{F}_{i,j} = F_{i,j} \setminus ([0,1] \times b(\{-1,1\} \times (-1,1)))$ and $\partial_0 = (\partial_{in} F_{00}) \cup ([0,1] \times b(\{-1,1\} \times \{-1,1\}))$, which is part of the boundary of every $\widetilde{F}_{i,j}$. Each of the embeddings $\widetilde{F}_{i,j} \hookrightarrow \widetilde{F}_{k,l}$ may be considered as composing $\widetilde{F}_{i,j}$ with a cobordism (the complement $\widetilde{F}_{k,l} \setminus \mathrm{int}(\widetilde{F}_{i,j})$), and we record these cobordisms in Figure \ref{fig:4TrivrSpin}. We denote by $A, B, \ldots, F$ the boundary conditions that these cobordisms satisfy at their ends, and by $L$ the boundary conditions they satisfy on their free boundary.
\begin{figure}[h]
\begin{center}
\includegraphics[bb=0 0 245 233]{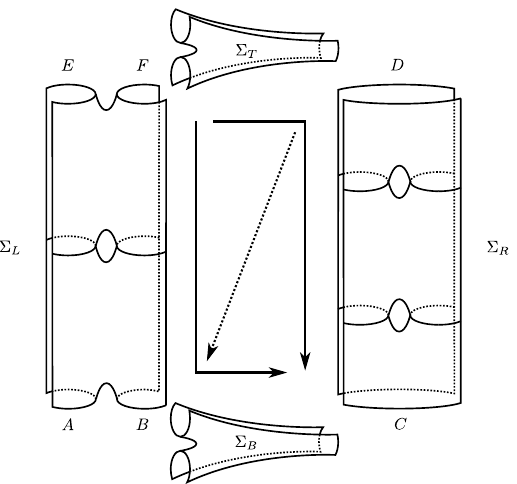}
\caption{}
\label{fig:4TrivrSpin}
\end{center}
\end{figure}

A trivialisation of the data $((s, \xi_s), b, (\bar{b}_i, \xi_{\bar{b}_i})_{i=1}^4)$ is then easily seen to be equivalent to the data of a cobordism with $\theta^2$-structure $\Sigma_\Delta : D \leadsto A \amalg B$ such that $\Sigma_B \circ \Sigma_\Delta \cong \Sigma_L$ and $\Sigma_\Delta \circ \Sigma_T \cong \Sigma_R$ as surfaces with $\theta^2$-structure, relative to $(E \amalg F) \cup L$ and $D \cup L$ respectively. Consider the commutative square
\begin{equation*}
\xymatrix{
\theta^2(\Sigma_{1,1};A \cup B \cup D \cup L)/\Gamma^+(\Sigma_{1,1}) \ar[d]^{\Sigma_B \circ -}\ar[r]^-{-\circ \Sigma_T}& \theta^2(\Sigma_{1,2};A \cup B \cup E \cup F \cup L)/\Gamma^+(\Sigma_{1,2}) \ar[d]^{\Sigma_B \circ -}\\
\theta^2(\Sigma_{1,2};C \cup D \cup L)/\Gamma^+(\Sigma_{1,2}) \ar[r]^-{-\circ \Sigma_T}& \theta^2(\Sigma_{2,1};C \cup E \cup F \cup L)/\Gamma^+(\Sigma_{2,1}).
}
\end{equation*}
By Corollary \ref{cor:GluingBij} each of these maps is a bijection, and in particular the square is cartesian. As $\Sigma_R \circ \Sigma_T = \Sigma_B \circ \Sigma_L$, it follows by the cartesian property that there exists a $\Sigma_\Delta \in \theta^2(\Sigma_{1,1};A \cup B \cup D \cup L)/\Gamma^+(\Sigma_{1,1})$ such that $\Sigma_\Delta \circ \Sigma_T = \Sigma_L$ and $\Sigma_B \circ \Sigma_\Delta = \Sigma_R$, as required. The argument for stabilisation maps of type $\beta$ is the same.
\end{proof}

%% file: chap3.tex
\section{Applications of homology stability for framed surfaces}

The methods of Galatius, Madsen, Tillmann and Weiss \cite[Section 7]{GMTW} take the homological stability theorem for framed surfaces and imply a homology equivalence
$$\colim_{g \to \infty} \mathcal{M}^{\fr}(\Sigma_{g,1+1}) =: \mathcal{M}^{\fr}(\Sigma_\infty) \lra \Omega^\infty \mathbf{S}^{-2} \simeq \Omega^2 Q(S^0)$$
where the colimit is formed by gluing on framed tori with two boundary components.
\begin{cor}
The moduli space of framed surfaces $\mathcal{M}^{\fr}(\Sigma_{g,b};\delta)$ has the integral homology of $\Omega^2 Q(S^0)$ in degrees $6* \leq 2g-8$. In particular it has trivial rational homology in these degrees. The same is true for the framed mapping class group $\Gamma^{\fr}(\Sigma_{g, b};\xi)$.
\end{cor}

\begin{cor}
The abelianisation of $\Gamma^{\fr}(\Sigma_{g,b};\xi)$ is $\bZ/24$, as long as $g \geq 7$.
\end{cor}
\begin{proof}
The abelianisation is simply the first integral homology of this group, which for $g \geq 7$ coincides with the first integral homology of $\Omega^\infty \mathbf{S}^{-2}$. By Hurewicz' theorem this is the same as $\pi_1(\Omega^\infty \mathbf{S}^{-2}) = \pi_3^s$, the third stable stem, which is well known to be cyclic of order 24.
\end{proof}

A diffeomorphism $\varphi : \Sigma_{g,1} \to \Sigma_{g, 1}$ that preserves a framing $\xi$ (up to isomorphism) represents an element $[\varphi] \in \Gamma^{\fr}(\Sigma_{g, 1};\xi)$ in the framed mapping class group, and one may ask what class it represents in the abelianisation $\bZ/24$.

The mapping torus $\Sigma_{g,1} \to M_\varphi \to S^1$ is a smooth 3-manifold with boundary $\partial \Sigma_{g, 1} \times S^1$, and we may give $M_\varphi$ a stable framing by choosing to take the bounding framing of $S^1$ and the framing $\xi$ along the fibres. On the boundary this framing of $\partial \Sigma_g^1 \times S^1$ bounds a framed solid torus by filling in a trivially framed disc in the $S^1$ direction. Gluing in such a framed solid torus in we obtain a closed framed 3-manifold $M^3$, which represents an element of $\pi_3^s \cong \bZ/24$.

It is well known that the element of $\bZ/24$ represented by a framed 3-manifold $M$ may be computed as follows. $M$ is in particular a Spin 3-manifold, and any closed Spin 3-manifold bounds a Spin 4-manifold, so let $M^3 = \partial W^4$. The first Pontrjagin class of $W$ may be represented by a cocycle which is identically zero on $\partial W = M$ as this is framed, so gives an element $p_1 \in H^4(W, M;\bZ)$. We may then form
$$-\frac{1}{48}\langle p_1, [W, M] \rangle \in \bQ.$$
A different choice $W'$ of Spin manifold means that $W \cup W'$ is a closed Spin manifold, and the two invariants defined above differ by $-\frac{1}{48}\langle p_1, [W \cup W'] \rangle$, which is an integer (it is half the $\hat{A}$-genus of the closed Spin manifold $W \cup W'$, which is an integer by Rokhlin's theorem). Thus we obtain a well defined element of $\bQ/\bZ$ which lies in the subgroup generated by $2/48$, which is isomorphic to $\bZ/24$. Hence, in principle we have described the map $\Gamma^{\fr}(\Sigma_{g,1};\xi) \to \bZ/24$, though in practice it is a matter of some difficulty to evaluate the above characteristic number.

%% file: chap4.tex
\section{The moduli spaces of $\Pin^\pm$ surfaces}\label{sec:PinSurfaces}

In the introduction we described how to a map $\theta : \X \to BO(2)$ we can associate a moduli space of surfaces with $\theta$-structure, so to define the moduli spaces of $\Pin$-surfaces it is enough to describe this map. This is complicated by the fact that there are two forms of a $\Pin(2)$ group, $\Pin^+(2)$ and $\Pin^-(2)$, which are most easily distinguished by saying that they are the central extensions of $O(2)$ classified by the classes $w_2$ and $w_2 + w_1^2$ in $H^2(BO(2); \bZ/2)$ respectively. Thus there are extensions
$$\bZ/2 \lra \Pin^\pm(2) \overset{\rho^\pm}\lra O(2)$$
and so principal fibrations
\begin{equation}\label{eq:PinPMDefinition}
B\bZ/2 \lra B\Pin^\pm(2) \overset{\theta^\pm}\lra BO(2).
\end{equation}
Let $\gamma_2^{\Pin^\pm} \to B\Pin^\pm(2)$ denote the pullback of the tautological bundle along $\theta^\pm$.

\begin{defn}
For each choice of sign, the \emph{moduli space of $\Pin^\pm$ surfaces} of topological type $F$ and boundary condition $\delta : TF\vert_{\partial F} \to \gamma_2^{\Pin^\pm}$, denoted $\mathcal{M}^{\Pin^\pm}(F;\delta)$, is that associated to the fibration $\theta^\pm$.
\end{defn} 

Much of the basic theory is just as it was for $r$-Spin structures. There is a fibration
$$\Bun_\partial(TF, \gamma^{\Pin^\pm}_2;\delta) \lra \Bun_\partial(TF, \gamma_2;[\delta]),$$
given by compostion with $\gamma_2^{\Pin^\pm} \to \gamma_2$, where $[\delta]$ denotes the composition $TF\vert_{\partial F} \overset{\delta}\to \gamma_2^{\Pin^\pm} \to \gamma_2$. The fibres of this map are either empty (if $F$ does not admit a $\Pin^\pm$-structure compatible with $\delta$), or else are non-canonically (in the same way as for $r$-Spin structures) homotopy equivalent to $\map_*(F/\partial F, B\bZ/2)$. Hence if $\mathcal{M}^{\Pin^\pm}(F;\delta)$ is non-empty there is a principal fibration 
\begin{equation}\label{eq:FibPin}
\map_*(F/\partial F, B\bZ/2) \lra \mathcal{M}^{\Pin^\pm}(F;\delta) \lra \mathcal{M}^{O(2)}(F;[\delta]).
\end{equation}

Both $\Pin^+(2)$ and $\Pin^-(2)$ become isomorphic to $\Spin(2)$ when restricted to covering groups of $SO(2)$. Hence if a surface is orientable and has boundary, a $\Pin^\pm$-structure on it is nothing but a (2-)Spin structure. As we have dealt with this situation in the previous sections, from now on we will assume that we only consider non-orientable surfaces.

\subsection{$\Pin^\pm$ mapping class groups}
We consider the long exact sequence on homotopy groups associated to the fibration (\ref{eq:FibPin}), based at $\xi \in \mathcal{M}^{\Pin^\pm}(F;\delta)$. We may identify the fibre through $\xi$ with $\map_*(F/\partial F, B\bZ/2)$, and if $\partial F$ is non-empty this space is homotopy-discrete, so we obtain an exact sequence on homotopy groups
$$0 \lra \Gamma^{\Pin^\pm}(F;\xi) \lra \Gamma(F) \lra H^1(F, \partial F;\bZ/2) \lra \pi_0(\mathcal{M}^{\Pin^\pm}(F;\delta)) \lra *$$
where $\Gamma^{\Pin^\pm}(F;\xi) := \pi_1(\mathcal{M}^{\Pin^\pm}(F;\delta), \xi)$. This identifies the $\Pin^\pm$ mapping class group as the subgroup of the unoriented mapping class group consisting of those diffeomorphisms which fix the $\Pin^\pm$-structure $\xi$ up to isomorphism.

If $\partial F$ is empty, we obtain instead an exact sequence
$$0 \lra \bZ/2 \lra \Gamma^{\Pin^\pm}(F;\xi) \lra \Gamma(F) \lra H^1(F;\bZ/2) \lra \cdots$$
which identifies the $\Pin^\pm$ mapping class group with an extension by $\bZ/2$ of the subgroup of the unoriented mapping class group of elements fixing the $\Pin^\pm$-structure $\xi$ up to isomorphism. This is analogous to the 2-Spin mapping class group in the case of oriented surfaces.

\subsection{The set of $\Pin^\pm$-structures}
In order to apply the results of \cite{R-WResolution}, we must compute the sets of components $\pi_0(\mathcal{M}^{\Pin^\pm}(F;\delta))$, at least when the genus of $F$ is large. This is the same as the quotient of the set of $\Pin^\pm$-structures
$$\Pin^\pm(F,\delta) := \pi_0(\Bun_\partial(TF, \gamma_2^{\Pin^\pm};\delta)),$$
by the action of the unoriented mapping class group $\Gamma(F)$. The set $\Pin^\pm(F,\delta)$ may be identified with the set of homotopy classes of lifts in the diagram
\begin{equation*}
\xymatrix{
& B\bZ/2 \ar[d]\\
\partial F \ar[r]^-{\delta}\ar@{^(->}[d]& B\Pin^\pm(2)\ar[d]^-{\theta^{\pm}}\\
F \ar[r]^-{[\ell]} \ar@{-->}[ru] & BO(2)
}
\end{equation*}
so is either empty or has a free and transitive acton of $H^1(F, \partial F ; \bZ/2)$.

As in the $r$-Spin case, we choose once and for all a $\Pin^\pm$-structure on $\bR^2$, which induces a \emph{standard} $\Pin^\pm$-structure on any framed rank 2 vector bundle. On the non-orientable surface $S_{n, b+1}$ having some $\Pin^\pm$-structure $\xi$ with boundary condition $\delta$, we choose a marked point on each boundary component, which has a canonical framing coming from the orientation of the boundary and the inwards pointing normal vector. After perhaps changing $\delta$ to an isomorphic boundary condition, we may assume that it is standard at each marked point with respect to this framing.

Let $\{ a_1, \ldots, a_n, \partial_0, \ldots, \partial_b \}$ be the collection of simple closed curves as shown in Figure \ref{fig:NonorientableCycles}, and $\{r_1, \ldots, r_b\}$ be the collection of simple arcs. We also write $a_i + a_j$ and $a_i+a_j+a_k+a_l$ for the simple closed curves shown, indicating the homology class they represent.
\begin{figure}
\begin{center}
\includegraphics[bb=0 0 353 205]{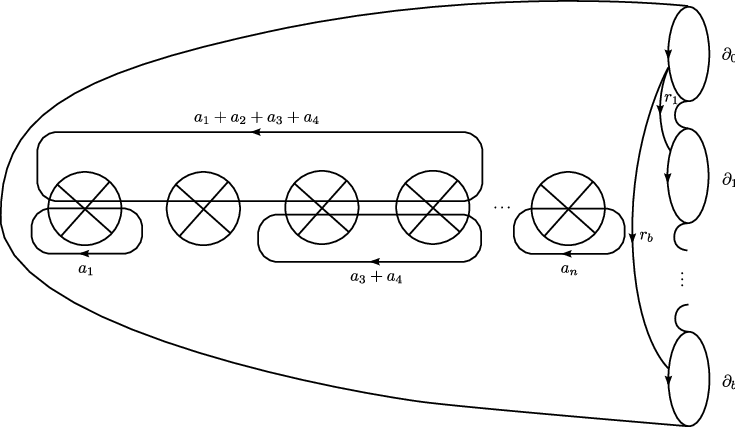}
\end{center}
\caption{Curves and arcs on $S_{n, b+1}$.}
\label{fig:NonorientableCycles}
\end{figure}
There is a presentation of the first integral homology of $S_{n, b+1}$ as
$$H_1(S_{n, b+1};\bZ) = \bZ\langle a_1, \ldots, a_n, \partial_0, \ldots, \partial_b \rangle/\langle \partial_0 + \cdots + \partial_b + 2(a_1 + \cdots + a_n) \rangle.$$

\begin{defn}
Given a $\Pin^+$-structure $\xi$ on $F$, along any simple closed curve $a$ or simple arc $r$ between standard marked points we have a canonical reduction of structure group from $\Pin^+(2)$ to $\Pin^+(1) = O(1) \times \bZ/2$. The first factor detects orientablity of the loop or arc. Taking the value of the monodromy along $a$ or $r$ on the second factor and adding 1 defines $q_\xi(a)$ or $q_\xi(r) \in \bZ/2$.

Given a $\Pin^-$-structure $\xi$ on $F$, along any simple closed curve $a$ or simple arc $r$ between standard marked points we have a canonical reduction of structure group from $\Pin^-(2)$ to $\Pin^-(1) = \bZ/4$. Adding 2 to the monodromy along $a$ or $r$ defines $q_\xi(a)$ or $q_\xi(r) \in \bZ/4$. The reduction of $q_\xi(x)$ modulo 2 detects orientability of the loop or arc.

We write $\Mon(a)$ or $\Mon(r)$ for the monodromy (in $\bZ/2$ or $\bZ/4$) around a curve $a$ or arc $r$ between standard marked points.
\end{defn}

The convention of adding the monodromy around a disc to the actual monodromy in this definition is as a normalisation: now $q_\xi(a)=0$ if and only if $a$ bounds a disc with $\Pin^\pm$-structure. We leave the proof of the fact that the monodromy around a disc is $1$ in the $\Pin^+$ case and $2$ in the $\Pin^-$ case to the reader: it is an interesting exercise in characteristic classes.

\begin{rem}[Boundary conditions]
For each oriented boundary component $\partial_i$ of a surface $F$ with $\Pin^\pm$-structure $\xi$, we may evaluate $q_\xi$ on the simple closed curve $\partial_i$ to get $\delta_i$. For $\Pin^+$-structures this lies in $\bZ/2$ and for $\Pin^-$-structures it lies in $2\bZ/4 = \bZ/2$ as the curve is orientation-preserving. This allows us to identify boundary conditions on $S_{n, b+1}$ for either of these tangential structures with elements of $(\bZ/2)^{b+1}$, and $\delta_i=0$ precisely when this boundary component bounds a disc.
\end{rem}

There is a map $e^+ : \Pin^+(S_{n, b+1};\delta) \to (\bZ/2)^{n+b}$ (or to $(\bZ/2)^{n}$ if the surface is closed) given by $e^+(\xi) = (q_\xi(a_1), \ldots, q_\xi(a_n), q_\xi(r_1), \ldots, q_\xi(r_b))$. Similarly, there is a map $e^-: \Pin^-(S_{n, b+1};\delta) \to (\bZ/4)^{n+b}$ by evaluation on the $a_i$ and $r_i$.

\begin{prop}\label{prop:PinPMAsSet}
If the set $\Pin^+(S_{n, b+1};\delta)$ is non-empty, $e^+$ is a bijection. If the set $\Pin^-(S_{n, b+1};\delta)$ is non-empty, $e^-$ is a bijection onto the subset $\{ [1],[3] \}^n \times \{[0], [2]\}^{b}$.
\end{prop}
\begin{proof}
Certainly both maps are injective: once a $\Pin^\pm$-structure $\xi$ is determined on $a_1, \ldots, a_n, r_1, \ldots, r_b$ and the boundary by $\delta$, it remains to give a $\Pin^\pm$-structure on a disc satisfying a certain boundary condition (which depends only on $\delta$). If this is possible, it is possible in at most one way, as the space of $\Pin^\pm$-structures on the disc is either empty or a torsor for $\Omega^2 B\bZ/2  \simeq *$.

The image of $e^-$ certainly lies in the subset $\{ [1],[3] \}^n \times \{[0], [2]\}^{b}$ due to the values a $\Pin^-$-structure may take around orientation-reversing and -preserving arcs. Now note that $\Pin^\pm(F;\delta)$ is either empty or a $H^1(F, \partial F;\bZ/2)$-torsor, so counting now implies that the maps are bijections.
\end{proof}

\begin{rem}
We should mention to what extent the functions $e^+$ and $e^-$ are unique. We have only defined them for boundary conditions with a choice of standard marked point on each boundary, and an ordering of the boundaries. Varying either of these choices will vary the functions.
\end{rem}

\begin{prop}\label{prop:PinTorsor}
Let $\xi \in \Pin^\pm(F;\delta)$ be a $\Pin^\pm$-structure and $g \in H^1(F, \partial F;\bZ/2)$ be a cohomology class. The torsor structure gives a new $\Pin^\pm$-structure $g\cdot \xi$ and on a simple closed curve or simple closed arc $a$ we have the formula
$$q_{g \cdot \xi}(a) = q_\xi(a) + g(a)\cdot \Mon(D^2),$$
where $\Mon(D^2)$ is $1$ in the $\Pin^+$ case and $2$ in the $\Pin^-$ case.
\end{prop}
\begin{proof}
Direct from the definition of the torsor structure.
\end{proof}

Using Proposition \ref{prop:PinPMAsSet}, we may study the action of $\Gamma(S_{n,b+1})$ on $\Pin^\pm(F;\delta)$ by studying its action in the ``coordinate system'' given by $e^\pm$. To do so we will require various formulae for computing with $q_\xi$.

\begin{lem}\label{lem:PossibleBoundaries}
Let $S_{n,1}$ be a non-orientable surface with $\Pin^\pm$-structure $\xi$. Then $\delta_0 = 0$ for $\Pin^-$, and $\delta_0 = n \,\, \mathrm{mod}\,\,  2$ for $\Pin^+$. Let $\Sigma_{g,1}$ be an orientable surface with $\Pin^\pm$-structure $\xi$. Then $\delta_0 = 0$ for both $\Pin^-$ and $\Pin^+$.

More generally, for a surface with multiple boundaries the same is true for $\sum \delta_i$.
\end{lem}
\begin{proof}
The value of $\delta_0$ is the value of $w_2 + w_1^2$ or $w_2$ respectively evaluated on $S_n$ or $\Sigma_g$. This is $0$ and $n \,\, \mathrm{mod} \,\, 2$ respectively on $S_n$, and always 0 on $\Sigma_g$.
\end{proof}

\begin{lem}
If $a$ and $b$ are disjoint simple closed curves (or a simple closed and a simple arc), and $a \# b$ is their oriented connected sum, then
$$q_\xi(a \# b) = q_\xi(a) + q_\xi(b).$$
\end{lem}
\begin{proof}
We may suppose that the connect sum is formed in a small disc on the surface where we have trivialised the $\Pin^\pm$-structure. Hence the difference in total monodromies of $a \sqcup b$ and $a \# b$ is the monodromy around a small disc, and the claim follows as $q_\xi$ is corrected from the actual monodromy by precisely the monodromy around a  disc.
\end{proof}

\begin{lem}\label{lem:QuadProperty}
Let $a$ and $b$ be simple closed curves (or a simple closed and a simple arc), and $c$ be a simple closed curve (or a simple closed arc).
\begin{enumerate}[(i)]
\item For $\Pin^-$, if the $\bZ/2$-homology classes satisfy $[c] = [a] + [b]$, then
$$q_\xi(c) = q_\xi(a) + q_\xi(b) + 2\cdot \langle a, b \rangle \in \bZ/4.$$
\item For $\Pin^+$, if the $\bZ/4$-homology classes satisfy $[c] = [a] + [b]$, then
$$q_\xi(c) = q_\xi(a) + q_\xi(b) + \langle a, b \rangle \in \bZ/2.$$
\end{enumerate}
If $b$ bounds a disc, we are asserting that $q_\xi(a)$ only depends on the ($\bZ/2$ or $\bZ/4$, respectively) homology class of $a$.
\end{lem}
\begin{proof}
Firstly, we perturb $a$ and $b$ so that they cross transversely. Near each intersection point we can cut out the intersection and glue the four incoming arcs together in pairs, to get a homologous collection of $N$ disjoint simple closed curves (with perhaps a single arc). Each time we do this, the total monodromy around all these curves does not change, so we still have total monodromy $\Mon(a) + \Mon(b)$. When we connect-sum together the $N$ disjoint curves, we end up with a single curve with monodromy
$$\Mon(a) + \Mon(b) + \Mon(D^2) \cdot (N -1),$$
as in the proof of the lemma above, and hence a single curve with invariant
$$q_\xi(a) + q_\xi(b) + \Mon(D^2) \cdot N.$$
We now simply remark that during the process of eliminating intersection points, the value of
$$\#\{\text{components}\} + \#\{\text{intersection points}\} \in \bZ/2$$
is constant, so in particular $N \equiv \langle a, b \rangle$ modulo 2, and the curve we have constructed has invariant $q_\xi(a) + q_\xi(b) + \Mon(D^2) \cdot \langle a, b \rangle$.

It is now enough to prove the homology invariance of $q_\xi$: if $a$ and $c$ are ($\bZ/4$ or $\bZ/2$, respectively) homologous, then $q_\xi(a) = q_\xi(c)$. By taking the difference of $a$ and $c$, and performing the above maneuver to get a single simple closed curve, it is enough to show that if $a$ is homologically trivial then $q_\xi(a)=0$.

To prove this we use the classification of simple closed curves in a surface (which is of course a simple consequence of the classification of surfaces). For the $\Pin^-$ case, a simple closed curve on a non-orientable surface can be trivial in $\bZ/2$-homology only if it bounds a subsurface, then by Lemma \ref{lem:PossibleBoundaries} the value of the invariant on it is zero. For the $\Pin^+$ case, a simple closed curve on a non-orientable surface can be trivial in $\bZ/4$-homology only if it bounds an orientable subsurface, then by Lemma \ref{lem:PossibleBoundaries} the value of the invariant is zero.
\end{proof}

We learnt the above results from the work of Degtyarev--Finashin \cite{Degtyarev-Finashin}, but have adapted the results and proofs in the form that is most convenient to us. Using the above lemma one can extend $q_\xi$ uniquely to a function on the first $\bZ/2$- or $\bZ/4$-homology of $S_{n, b+1}$, and the lemma shows that $q_\xi$ is a quadratic refinement of the intersection form.

We wish to compute the action of the unoriented mapping class group $\Gamma(S_{n, b+1})$ on the sets $\Pin^\pm(S_{n, b+1}; \delta)$. Recall that the unoriented mapping class group is generated by Dehn twists and crosscap slides \cite[Theorem 2]{Lickorish}, and the following lemma describes the action of a crosscap slide on homology. %By the following lemma, it will be enough to consider just Dehn twists.

\begin{lem}\label{lem:crosscapslide}
The crosscap slide of the crosscap with core $a_1$ over the crosscap with core $a_2$ in Figure \ref{fig:NonorientableCycles}, supported inside an embedded $S_{2,1}$, induces the map
\begin{align*}
a_1 &\longmapsto -a_1\\
a_2 &\longmapsto a_2 +2\cdot a_1\\
a_i &\longmapsto a_i \text{ if $i > 2$}
\end{align*}
on integral homology.
\end{lem}
\begin{proof}
This is immediate from the definition of the crosscap slide.
\end{proof}

We will require the following formulae for the action of Dehn twists on the set of $\Pin^\pm$-structures. Let $a$ be a simple closed curve, $\tau_a$ the Dehn twist around $a$, and let $x$ is a simple arc or simple closed curve. For $\Pin^+$-structures
$$q_{\tau_a^*\xi}(x) =  q_\xi(x) + \langle a,x \rangle(q_\xi(a) - 1) \in \bZ/2,$$
and for $\Pin^-$-structures
$$q_{\tau_a^*\xi}(x) = q_\xi(x) + \langle a,x \rangle(q_\xi(a) - 2) \in \bZ/4.$$
Both equations come from the standard action of Dehn twists on homology, $\tau_a(x) = x + \langle a, x \rangle\cdot a$, and Lemma \ref{lem:QuadProperty}.

\subsection{Diffeomorphism classes of $\Pin^+$-structures}\label{sec:DiffClassesPinPlusStruct}

Once we have a bijection $e^+ : \Pin^+(S_{n, b+1};\delta) \to (\bZ/2)^{n+b}$, we define a function
$$A : \Pin^+(S_{n,b+1};\delta) \lra \bZ/2$$
by the formula $\xi \mapsto \sum_{i=1}^n q_\xi(a_i)$.

\begin{prop}
The function $A$ is $\Gamma(S_{n,b+1})$-invariant.
\end{prop}
\begin{proof}
It is enough to show that it is invariant under Dehn twists and a single crosscap slide, as these generate the mapping class group by \cite[Theorem 2]{BC}. 

Let $a$ be an orientation preserving simple closed curve represented in homology by $\sum A_i a_i + \sum \lambda_j \partial_j \in H_1(S_{n, b+1};\bZ)$, so $\langle a, a\rangle = \sum A_i = 0$. We may then compute
$$\sum_{i=1}^n  q_{\tau_a^*\xi}(a_i) = \sum_{i=1}^n q_\xi( a_i + \langle a_i, a \rangle\cdot a) = \sum_{i=1}^n q_\xi(a_i) + \sum_{i=1}^n A_i (1+q_\xi(a)) = \sum_{i=1}^n q_\xi(a_i) \in \bZ/2$$
so $A$ is invariant under Dehn twists.%a diffeomorphism invariant of the $\Pin^+$-structure $\xi$.

Let $c$ be the crosscap slide described in Lemma \ref{lem:crosscapslide}, and calculate using Lemma \ref{lem:QuadProperty} that
\begin{align*}
q_{c^*\xi}(a_1) &= q_\xi(-a_1) = q_\xi(a_1)+1\\
q_{c^*\xi}(a_2) &= q_\xi(a_2+2\cdot a_1) = q_\xi(a_2)+1\\
q_{c^*\xi}(a_i) &= q_\xi(a_i) \text{ if $i > 2$}.
\end{align*}
Thus $A$ is invariant under this crosscap slide, and hence under all diffeomorphisms.
\end{proof}

\begin{prop}\label{prop:PinPlusInvariant}
The induced map $A : \Pin^+(S_{n,b+1};\delta)/\Gamma(S_{n,b+1}) \to \bZ/2$ is a bijection if the source is non-empty and $n \geq 3$. It is a surjection for $n \geq 1$.
\end{prop}
\begin{proof}
We use the bijection $e^+$ to identify $\Pin^+(S_{n,b+1};\delta)$ with $(\bZ/2)^{n+b}$, and write elements as $(A_1, \ldots, A_n, R_1, \ldots, R_b)$. If we take the simple closed curve $a_i + a_j$ from Figure \ref{fig:NonorientableCycles} (when $n \geq 2$), we have
$$q_{\tau_{a_i + a_j}^*\xi}(a_s) = q_\xi(a_s + \langle a_i + a_j, a_s \rangle \cdot (a_i + a_j)) =
\begin{cases}
q_\xi(a_s) & s \neq i, j\\
1+q_\xi(a_j) & s = i \\
1+q_\xi(a_i) & s = j.
\end{cases}$$
Thus given an element $(A_1, \ldots, A_n) \in (\bZ/2)^n$, we may permute a pair of entries and add 1 to both, via Dehn twists. In particular, we may remove pairs of 1's, and we may move a single 1 to be at any position (if $n \geq 3$). Thus we may reduce any element to the form
$$(0, 1, 0, \ldots, 0, R_1, \ldots, R_b) \,\,\text{or}\,\, (0, 1, 1, 0 \ldots, 0, R_1, \ldots, R_b)$$
as long as $n \geq 3$.

Suppose that $R_i=1$. If $\delta_i=0$ then twisting along $\partial_i$ gives
$$q_{\tau_{\partial_i}^*\xi}(r_i) = q_\xi(r_i) + q_\xi(\partial_i)-1 = 0$$
and does not affect the other coordinates. If $\delta_i=1$ then the simple closed curve $a_1 + a_2$ may be connect-summed to the simple closed curve $\partial_i$ to give a curve $x$ which intersects $a_1$, $a_2$ and $r_i$ once, the other curves and arcs not at all, and has $q_\xi(x) = 0$. Then
$$q_{\tau_{x}^*\xi}(r_i) = q_\xi(r_i) + q_\xi(x)-1 = 0$$
so twisting around this sets $R_i$ to 0. Doing this for each $i$ allows us to reduce to an element of the form
$$(0, 1, 0, \ldots, 0) \,\,\text{or}\,\, (0, 1, 1, 0, \ldots, 0)$$
as long as $n \geq 3$, so there are at most two orbits. On the other hand, the invariant $A$ distinguishes these two elements, so there are precisely two orbits, distinguished by $A$.
%If $n=2$ then $(0,0)$, $(0,1)$, $(1,0)$ represent all orbits, so there are at most 3.
\end{proof}

We also make the following observation about the behaviour of $A$ with respect to the $H^1(F, \partial F;\bZ/2)$-torsor structure on $\Pin^+(F;\delta)$, which is immediate from the formula for $A$,
\begin{equation}\label{eq:TorsorPinPlus}
A(g \cdot \xi) = A(\xi) + g\left(\sum_{i=1}^n a_i \right).
\end{equation}

\subsection{Diffeomorphism classes of $\Pin^-$-structures}\label{sec:DiffClassesPinMinusStruct}

Once we have the map $e^- : \Pin^-(S_{n, b+1};\delta) \to (\bZ/4)^{n+b}$ which is a bijection onto the subset $\{ [1],[3] \}^n \times \{[0], [2]\}^{b}$, we define a function
$$A : \Pin^-(S_{n,b+1};\delta) \lra \bZ/4$$
by the formula
$$\xi \mapsto \#\{i \,\,\vert\,\,q_\xi(a_i) = 1\} - \sum_{j=1}^b \delta_i \cdot \frac{q_\xi(r_i)}{2},$$
bearing in mind that $r_i$ is an orientation-preserving arc so $q_\xi(r_i)$ is $0$ or $2$, and $\tfrac{q_\xi(r_i)}{2}$ is then defined to be $0$ or $1$.

\begin{prop}
The function $A$ is $\Gamma(S_{n,b+1})$-invariant.
\end{prop}
\begin{proof}
It is enough to show that it is invariant under Dehn twists and a single crosscap slide, as these generate the mapping class group by \cite[Theorem 2]{BC}. 

%We only need to show that it is invariant under Dehn twists.
We first treat Dehn twists. Let $x$ be an orientation-preserving simple closed curve represented in homology by $\sum X_i a_i + \sum \lambda_j \partial_j$, and $\xi$ a $\Pin^-$-structure. Then
$$q_{\tau_x^*\xi}(a_i) = q_\xi(a_i) + X_i (q_\xi(x) -2)$$
and
$$q_{\tau_x^*\xi}(r_j) = q_\xi(r_j) + (\lambda_j-\lambda_0)(q_\xi(x) -2)$$
so if $q_\xi(x)=2$ the proposed invariant is trivially preserved. If $q_\xi(x)=0$ then
$$\sum \delta_i \cdot \frac{q_{\tau_x^*\xi}(r_i)}{2} - \sum \delta_i \cdot \frac{q_\xi(r_i)}{2}  = -\sum \delta_j(\lambda_j-\lambda_0).$$ 
Using Lemma \ref{lem:PossibleBoundaries} we have $\sum \delta_i=0$, and so
$$-\sum_{j=1}^r \delta_j(\lambda_j-\lambda_0) = -\sum_{j=0}^r \delta_j \lambda_j = -q_\xi(x)+ \sum X_i q_{\tau_x^*\xi}(a_i) = \sum X_i q_{\tau_x^*\xi}(a_i),$$
which is $\sum X_i(q_\xi(a_i)-2X_i) = \sum X_i q_\xi(a_i) - 2\langle x,x \rangle = \sum X_i q_\xi(a_i)$, as $\langle x, x \rangle=0$ because $x$ is orientation-preserving.

On the other hand, $q_{\tau_x^*\xi}(a_i)=1$ if and only if $q_\xi(a_i) = 1 + 2X_i$, which falls into two disjoint cases:
\begin{enumerate}[(i)]
    \item either $q_\xi(a_i) = 1$ and $X_i$ = 0,
	\item or $q_\xi(a_i) = 3$ and $X_i=1$.
\end{enumerate}
Thus the difference $\#\{i \,\,\vert\,\,q_\xi(a_i)=1\} - \#\{i \,\,\vert\,\,q_\xi(a_i)=1+2X_i\}$ is the same as the difference $\#\{i \,\,\vert\,\,q_\xi(a_i)=1, X_i=1\} - \#\{i \,\,\vert\,\,q_\xi(a_i)=3, X_i=1\}$
which is $\sum X_i q_\xi(a_i)$, as required.

If $c$ is the crosscap slide described in Lemma \ref{lem:crosscapslide} then we calculate using Lemma \ref{lem:QuadProperty} that
\begin{align*}
q_{c^*\xi}(a_1) &= q_\xi(-a_1) = 2-q_\xi(a_1)\\
q_{c^*\xi}(a_2) &= q_\xi(a_2+2\cdot a_1) = q_\xi(a_2) + 2\cdot q_\xi(a_1)+2\\
q_{c^*\xi}(a_i) &= q_\xi(a_i) \text{ if $i > 2$}.
\end{align*}
Now $q_\xi(a_i) = [1]$ or $[3]$ in $\bZ/4$, so $2-q_\xi(a_1) = q_\xi(a_1)$ and $2\cdot q_\xi(a_1)+2 = 0$, so $c$ acts trivially on $\Pin^-(S_{n,b+1};\delta)$ and in particular preserves $A$.
\end{proof}

\begin{prop}\label{prop:PinMinusInvariant}
The induced map $A : \Pin^-(S_{n,b+1};\delta)/\Gamma(S_{n,b+1}) \to \bZ/4$ is a bijection if the source is non-empty and $n \geq 3$.
\end{prop}
\begin{proof}
Let us use the bijection $e^-$ to identify $\Pin^-(S_{n,b+1};\delta)$ with $\{[1], [3]\}^n \times \{[0], [2]\}^b$, and write elements as $(A_1, \ldots, A_n, R_1, \ldots, R_b)$. Note that
$$q_{\tau_{a_i + a_j}^*\xi}(a_s) = q_\xi(a_s + \langle a_i + a_j, a_s \rangle \cdot (a_i + a_j)) =
\begin{cases}
q_\xi(a_s) & s \neq i, j\\
q_\xi(a_j) & s = i \\
q_\xi(a_i) & s = j.
\end{cases}$$
Thus given an element $(A_1, \ldots, A_n) \in (1+2\bZ/4)^n$ there is a diffeomorphism which permutes its entries arbitrarily. Note also that
$$q_{\tau_{a_i + a_j + a_k + a_l}^*\xi}(a_s) =
\begin{cases}
q_\xi(a_s) & s \neq i, j, k, l\\
q_\xi(a_j+a_k+a_l) & s = i \\
q_\xi(a_i+a_k+a_l) & s = j \\
q_\xi(a_i+a_j+a_l) & s = k \\
q_\xi(a_i+a_j+a_k) & s = l \\
\end{cases}$$
This allows us to replace an occurrence of $(3,3,3,3)$ in $(A_1, \ldots, A_n)$ by $(1,1,1,1)$. 

Suppose that $R_i=2$. If $\delta_i=0$ then twisting along $\partial_i$ sets $R_i$ to 0. As $n \geq 3$ there is a pair of basis elements $a_j$, $a_k$ such that $\{A_j, A_k\} = \{1,1\}$ or $\{3,3\}$. If $\delta_i=2$ the simple closed curve $a_j+a_k$ may be connect-summed to the simple closed curve $\partial_i$ to give a simple closed curve $x$ which intersects $a_j$, $a_k$ and $r_i$ once, the other curves not at all, and has $q_\xi(x)= q_\xi(a_j) + q_\xi(a_k) + \delta_i = 0$. Then
$$q_{\tau_{x}^*\xi}(r_i) = q_\xi(r_i) + q_\xi(x) - 2 = 0$$
so twisting around this sets $R_i$ to 0. This shows how to reduce to $R_1=\cdots=R_r=0$. Thus for $n \geq 3$ there are at most 4 orbits. On the other hand, the invariant $A$ distinguishes the four elements
$$( 3, \ldots, 3, 0, \ldots, 0), (1,, 3, \ldots, 3, 0, \ldots, 0), \dots, (1, 1, 1, 3, \ldots, 3, 0, \ldots, 0)$$
and so there are precisely four orbits, distinguished by $A$.
\end{proof}

We also make the following observation about the behaviour of $A$ with respect to the $H^1(F, \partial F;\bZ/2)$-torsor structure on $\Pin^-(F;\delta)$, which is immediate from the formula for $A$,
\begin{equation}\label{eq:TorsorPinMinus}
A(g \cdot \xi) = A(\xi) + \sum_{i=1}^n g(a_i) \cdot(q_\xi(a_i)-2) - \sum_{j=1}^b g(r_j)\cdot \delta_j .
\end{equation}

\begin{rem}
The calculation of the sets $\Pin^\pm(S_n)/\Gamma(S_n)$ has also been attempted by D\k{a}browski and Percacci \cite{DP}. They correctly find that $\Pin^+(S_n)/\Gamma(S_n)$ is empty if $n$ is odd, and that $\Pin^-(S_n)/\Gamma(S_n)$ has four elements for $n \geq 3$ (they do not make explicit this requirement on $n$, but it is required for all values of their invariants to be attained). However they \emph{incorrectly} claim that $\Pin^+(S_n)/\Gamma(S_n)$ has three elements when $n$ is even: by our Proposition \ref{prop:PinPlusInvariant} it consists of two elements as long as $n \geq 3$. 

The problem in \cite{DP} seems to be in the definition of the invariant $\ell$. As the authors state, the orientation involution $J$ on the orientation cover $\Sigma_{n-1}$ lifts canonically to an involution $TJ$ of the frame bundle $F_\Sigma$. If this lifts to an involution $\widetilde{T}J$ of a $\Pin^\pm$-structure $\widetilde{F}_\Sigma \to F_\Sigma$ then composing $\widetilde{T}J$ with multiplication with $-1 \in \Pin^\pm(2)$ gives another involution of $\widetilde{F}_\Sigma$, and the authors propose to label these two involutions by $\ell \in \{0,1\}$. But there is no canonical way to choose which should be labelled 0 and which should be labelled 1, so the ``invariant'' $\ell$ is not defined.
%In the language of this paper, the two $\Pin^\pm$-structures given by two choices of involution which differ by $-1 \in \Pin^\pm(2)$ have quadratic forms which differ by multiplication by the first Stiefel--Whitney class under the torsor structure.
\end{rem}

\subsection{Gluing $\Pin^\pm$ surfaces}

As in \S \ref{sec:GluingThetaRSurfaces}, we require a formula for the effect on the invariant $A$ of gluing together surfaces. For our purposes we only need such a formula for gluing together cobordisms of the form $S_{n, 1+1}$, and so we shall not investigate the existence of a formula for gluing more general cobordisms.

If $(F, \delta)$ and $(F', \delta')$ are cobordisms with $\Pin^\pm$ boundary conditions and there is an identification $\psi : \partial_{in} F' \cong \partial_{out} F$ of boundary components such that $\psi^*(\delta'\vert_{\partial_{in}F'}) = \delta\vert_{\partial_{out} F}$, then we obtain a gluing map
\begin{eqnarray*}
G_\psi: \Pin^\pm(F;\delta)/ \Gamma(F) \times \Pin^\pm(F';\delta')/ \Gamma(F') 
\lra \Pin^\pm(F \cup_\psi F';\delta \cup \delta)/\Gamma(F \cup_\psi F').
\end{eqnarray*}
We always suppose that each boundary components of each surface has a marked point with standard $\Pin^\pm$-structure, and that these are identified under $\psi$.

\begin{lem}\label{lem:PinAdditivity}
Consider gluing together two $\Pin^\pm$ cobordisms $S_{n, 1+1}$ and $S_{n', 1+1}$. 

In the $\Pin^+$ case, the function $A : \Pin^+(S_{n+n', 1+1}, \delta \cup \delta')/\Gamma(S_{n+n',1+1}) \to \bZ/2$ evaluated on the glued $\Pin^+$-structure $G_\psi(\xi, \xi')$ is $A(\xi) + A(\xi')$. 

In the $\Pin^-$ case, the function $A : \Pin^-(S_{n+n', 1+1}, \delta \cup \delta')/\Gamma(S_{n+n',1+1}) \to \bZ/4$ evaluated on the glued $\Pin^-$-structure $G_\psi(\xi, \xi')$ is $A(\xi) + A(\xi') + \delta\vert_{\partial_{in} S_{n', 1+1}}$.
\end{lem}
\begin{proof}
These formulas are immediate from the definitions of the invariant $A$, and in the $\Pin^-$ case the observation that when gluing two arcs together at a single end the monodromy adds, so the value of $q_\xi$ adds, \emph{but then two is added}.
\end{proof}

We also require the following lemma concerning more general gluings, though not any explicit formula for the effect on the invariant $A$.

\begin{lem}\label{lem:PinAdditivity2}
If we fix a (possibly orientable) $\Pin^\pm$-surface $F'$, then the map
$$G_\psi : \Pin^\pm(F;\delta)/ \Gamma(F) \lra \Pin^\pm(F \cup_\psi F';\delta \cup \delta')/\Gamma(F \cup_\psi F')$$
is a bijection as long as $F$ is non-orientable of genus $\geq 3$.
\end{lem}
\begin{proof}
Note both sides have the same cardinality (2 in the $\Pin^+$ case and 4 in the $\Pin^-$ case), and so it is enough to show that the map is surjective. For this we consider the commutative diagram
\begin{equation*}
\xymatrix{
{\Pin^\pm(F;\delta)} \ar[d]_-{G_\psi}\ar@{->>}[r] & {\Pin^\pm(F;\delta)/ \Gamma(F)} \ar[d]_-{G_\psi}\ar[r]^-A_-\sim& {\bZ/2 \,\, \text{or}\,\, \bZ/4} \ar[d]^{\text{induced}}\\
{\Pin^\pm(F \cup_\psi F';\delta \cup \delta')} \ar@{->>}[r]  & {\Pin^\pm(F \cup_\psi F';\delta \cup \delta')/\Gamma(F \cup_\psi F')} \ar[r]^-A_-\sim & {\bZ/2 \,\, \text{or}\,\, \bZ/4}
}
\end{equation*}
where the left vertical map is a map of torsors over the map
$$H^1(F, \partial F;\bZ/2) \lra H^1(F \cup_\psi F', \partial (F \cup_\psi F');\bZ/2),$$
and we claim that composition along the bottom of the diagram is surjective: it follows that the right-hand vertical map is surjective, which proves the lemma.

In the $\Pin^+$ case the formula (\ref{eq:TorsorPinPlus}) shows that the composition along the bottom of the diagram is surjective, as we may always find a $g \in H^1(F, \partial F;\bZ/2)$ that evaluates to $1$ on $\sum a_i$.

In the $\Pin^-$ case, observe that given a tuple of elements in $\{[1], [3]\}^3$ we may write any element of $\bZ/4$ as a sum of them using only coefficients 0 and 1. This observation along with the formula (\ref{eq:TorsorPinMinus}) shows that composition along the bottom of the diagram is surjective as long as $F$ has genus $\geq 3$.
\end{proof}

\subsection{Homological stability}\label{sec:HomStabPin}
We will now use the above results to apply the main theorem of \cite{R-WResolution} in the case of non-orientable surfaces. The proofs necessarily use terminology introduced in \cite{R-WResolution}, and we will not give these definitions again.

\begin{thm}
The moduli spaces of $\Pin^+$-surfaces exhibit homological stability. More precisely,
\begin{enumerate}[(i)]
	\item Any $\alpha(n): \mathcal{M}^{\Pin^+}(S_{n,b}) \to \mathcal{M}^{\Pin^+}(S_{n+2, b-1})$ is a homology isomorphism in degrees $4* \leq n-6$.
	\item Any $\beta(n): \mathcal{M}^{\Pin^+}(S_{n, b}) \to \mathcal{M}^{\Pin^+}(S_{n, b+1})$ is a homology isomorphism in degrees $4* \leq n - 6$. If one of the created boundary conditions is trivial, it is a split homology monomorphism in all degrees.
	\item Any $\gamma(n): \mathcal{M}^{\Pin^+}(S_{n, b}) \to \mathcal{M}^{\Pin^+}(S_{n, b-1})$ is a homology isomorphism in degrees $4* \leq n - 6$. If $b \geq 2$ it is a split homology epimorphism in all degrees, and if $b=1$ it is a homology epimorphism in degrees $4* \leq n - 2$.
	\item Any $\mu(n): \mathcal{M}^{\Pin^+}(S_{n, b}) \to \mathcal{M}^{\Pin^+}(S_{n+1, b})$ is a homology epimorphism in degrees $4* \leq n - 2$, and a homology isomorphism in degrees $4* \leq n - 6$.
\end{enumerate}
\end{thm}
\begin{proof}
In order to apply Theorems 8.2 and 12.4 of \cite{R-WResolution}, we must verify a collection of conditions: that $\Pin^+$-structures \emph{stabilise on $\pi_0$ at genus $h'$ for projective planes}, that they are \emph{$k'$-trivial for projective planes}, and that they \emph{stabilise on $\pi_0$ at genus $h$}. We claim that this is indeed so with $(h', k', h) = (3,3,3)$, so solving the recurrence relations of \cite[\S 7.5]{R-WResolution} with this data gives $H'(n) = \lfloor \tfrac{n-2}{4}\rfloor$, and so the stated stability ranges.

The path-components of the moduli space of $\Pin^+$-structures stabilise at genus $3$ both for projective planes and in general, by Proposition \ref{prop:PinPlusInvariant} and Lemma \ref{lem:PinAdditivity2}. Precisely, Lemma \ref{lem:PinAdditivity2} shows that all stabilisation maps are bijections for genus at least 3, and the second part of Proposition \ref{prop:PinPlusInvariant} shows that the invariant $A$ is surjective for genus at least 1, while the first part of Proposition \ref{prop:PinPlusInvariant} shows that the invariant $A$ is complete for genus at least 3: it follows that stabilisation maps starting in genus 2 are surjective on path-components. Then \cite[Proposition 7.7]{R-WResolution} implies that $\Pin^+$-structures are 3-trivial for projective planes.
\end{proof}

\begin{thm}
The moduli spaces of $\Pin^-$-surfaces exhibit homological stability. More precisely,
\begin{enumerate}[(i)]
	\item Any $\alpha(n): \mathcal{M}^{\Pin^-}(S_{n,b}) \to \mathcal{M}^{\Pin^-}(S_{n+2, b-1})$ is a homology isomorphism in degrees $5* \leq n-8$.
	\item Any $\beta(n): \mathcal{M}^{\Pin^-}(S_{n,b}) \to \mathcal{M}^{\Pin^-}(S_{n, b+1})$ is a homology isomorphism in degrees $5* \leq n - 8$. If one of the created boundary conditions is trivial, it is a split homology monomorphism in all degrees.
	\item Any $\gamma(n): \mathcal{M}^{\Pin^-}(S_{n, b}) \to \mathcal{M}^{\Pin^-}(S_{n, b-1})$ is a homology isomorphism in degrees $5* \leq n - 8$. If $b \geq 2$ it is a split homology epimorphism in all degrees, and if $b=1$ it is a homology epimorphism in degrees $5* \leq n-3$.
	\item Any $\mu(n): \mathcal{M}^{\Pin^-}(S_{n, b}) \to \mathcal{M}^{\Pin^-}(S_{n+1, b})$ is a homology epimorphism in degrees $5* \leq n - 3$, and a homology isomorphism in degrees $5* \leq n - 8$.
\end{enumerate}
\end{thm}
\begin{proof}
As in the last theorem, but we claim we have $(h', k', h)=(4,4,4)$. The recurrence relations of \cite[\S 7.5]{R-WResolution} with this data gives $H'(n) = \lfloor \tfrac{n-3}{5}\rfloor$, and so the stated stability ranges. Again we deduce this from Lemma \ref{lem:PinAdditivity2}, but cannot prove stability with $h'=3$ in this case because the map
$$\mu(2) : \pi_0(\mathcal{M}^{\Pin^-}(S_{2, 1})) \lra \pi_0(\mathcal{M}^{\Pin^-}(S_{3, 1}))$$
is not surjective: by Proposition \ref{prop:PinMinusInvariant} the target has cardinality 4 and is detected by the $A$-invariant, whereas the set $\Pin^-(S_{2,1};\delta)$ is in bijection with $\{[1],[3]\}^2$ and its $A$-invariant---which counts the number of $[1]$'s in a vector---can only take values $0$, $1$, or $2$, so is not surjective.
\end{proof}

\section{Applications of homology stability for surfaces with $\Pin^\pm$-structure}

Let us denote by $\MT{Pin^\pm}{2}$ the Thom spectrum of the virtual bundle $-\gamma_2^{\Pin^\pm} \to B\Pin^\pm(2)$. There is a natural comparison map
$$\alpha : \coprod_{n \geq 1}\mathcal{M}^{\Pin^\pm}(S_{n, 1};\delta) \lra \Omega^\infty \MT{Pin^\pm}{2}$$
where $\delta$ is a boundary condition which bounds a disc. The left hand side admits the structure of a topological monoid, under the ``pair of pants'' product, and Galatius and the author \cite{GR-W} have shown that this map is a group-completion. In particular, applying the group-completion theorem \cite{McDuff-Segal} we obtain a homology equivalence
\begin{equation}\label{eq:GroupCompletionPinMonoid}
\bZ \times \mathcal{M}^{\Pin^\pm}(S_{\infty}) \lra \Omega^\infty \MT{Pin^\pm}{2}.
\end{equation}
This gives an isomorphism from the group $\pi_0(\Omega^\infty \MT{Pin^+}{2})$ to $\bZ \times \bZ/2$ and from the group $\pi_0(\Omega^\infty \MT{Pin^-}{2})$ to $\bZ \times \bZ/4$. Let us denote by $\Omega^\infty_\bullet \MT{Pin^\pm}{2}$ those path components corresponding to $0$ on the $\bZ$ factor. Combining this with the homology stability theorem proved in this paper, we establish the following corollary.

\begin{cor}
The space $\mathcal{M}^{\Pin^\pm}(S_{n, b};\delta)$ has the integral homology of the infinite loop space $\Omega^\infty_\bullet \MT{Pin^\pm}{2}$ in degrees $4* \leq n-6$ in the case $\Pin^+$ and in degrees $5* \leq n-8$ in the case $\Pin^-$.
\end{cor}

Taking fundamental groups, we obtain the following computational corollary.

\begin{cor}
The abelianisation of $\Gamma^{Pin^+}(S_{n, b};\xi)$ is $\bZ/2$ for $n \geq 10$. The abelianisation of $\Gamma^{Pin^-}(S_{n, b};\xi)$ is $(\bZ/2)^3$ for $n \geq 13$.
\end{cor}
\begin{proof}
By the homological stability theorems for these groups, their abelianisation in this range coincides with the first homology of the spaces $\Omega^\infty_0 \MT{Pin^+}{2}$ and $\Omega^\infty_0 \MT{Pin^-}{2}$ respectively. By Hurewicz' theorem, this coincides with the first homotopy group of these spaces. These are computed in Appendix \ref{sec:HomotopyOfPinPM}.
\end{proof}

Let us briefly explain the relationship between our calculations
$$\pi_0(\MT{Pin^+}{2}) = \bZ \oplus \bZ/2 \quad\quad\quad \pi_0(\MT{Pin^-}{2}) = \bZ \oplus \bZ/4$$
and the calculation of Kirby and Taylor \cite{KT} of the bordism groups
$$\Omega_2^{\Pin^+} = \bZ/2 \quad\quad\quad \Omega_2^{\Pin^-} = \bZ/8.$$
There are natural maps of spectra $s^\pm: \MT{Pin^\pm}{2} \to \Sigma^{-2} \mathbf{MPin^\pm}$ given by taking Thom spectra of $B\Pin^\pm(2) \to B\Pin^\pm$, and it is easy to check that the fibre $\mathbf{F}$ of $s^\pm$ is connective and has $\pi_0(\mathbf{F})\cong\bZ$, and a generator of this group is given by $S^2$ with its unique $\Pin^\pm$-structure. Furthermore, the Euler characteristic is well-defined on $\pi_0(\MT{Pin^\pm}{2})$, giving a homomorphism $\chi : \pi_0(\MT{Pin^\pm}{2}) \to \bZ$. We obtain a diagram
\begin{equation*}
\xymatrix{
0 \ar[r]& {\bZ} \ar[r]^-{1 \mapsto S^2} \ar[rd]_-{1 \mapsto 2}& {\pi_0(\MT{Pin^\pm}{2})} \ar[r]^-{\pi_0(s^\pm)}\ar[d]^-\chi& {\Omega_2^{\Pin^\pm}} \ar[r]& 0\\
& & {\bZ}.
}
\end{equation*}
In the $\Pin^+$ case only surfaces of even Euler characteristic admit $\Pin^+$-structures, so $\chi$ is onto $2\bZ$ and hence $\chi/2$ gives a splitting of the top short exact sequence, giving $\pi_0(\MT{Pin^+}{2}) \cong \bZ \oplus\bZ/2$, as we have calculated.

In the $\Pin^-$ case all surfaces admit $\Pin^-$-structures, so $\chi$ is surjective and cannot be used to split the short exact sequence. Instead we see that the square
\begin{equation*}
\xymatrix{
{\pi_0(\MT{Pin^-}{2})} \ar[r]^-{\pi_0(s^-)}\ar[d]^-\chi& {\Omega_2^{\Pin^-}} \ar[d]^{M \mapsto \langle [M], w_2(M)\rangle} & \!\!\!\!\!\!\!\!\!\!\!\!\!\!\!\!\!\! \cong \bZ/8\\
{\bZ} \ar[r]^{1 \mapsto [1]}& {\bZ/2}.
}
\end{equation*}
is cartesian, so $\pi_0(\MT{Pin^-}{2}) \cong \bZ \oplus \bZ/4$ as we have calculated.

\subsection{Stable homology of the $\Pin^\pm$ mapping class groups}\label{sec:StableHomologyPin}

Let us briefly discuss the homology of the moduli spaces $\mathcal{M}^{\Pin^\pm}(S_\infty)$, which coincides with the homology of the stable mapping class group $\Gamma^{\Pin^\pm}(S_\infty)$. This will require certain calculations in homotopy theory which we have included as Appendix \ref{sec:HomotopyOfPinPM}. By Proposition \ref{prop:OddCohomologyPin} and \cite[\S 5.1]{R-W} there are equivalences of $\bZ[\tfrac{1}{2}]$-local spectra
\begin{equation}\label{eq:ZHalfHomology}
\MT{Pin^\pm}{2}\left [\tfrac{1}{2} \right ] \simeq \MT{O}{2}\left [\tfrac{1}{2} \right ] \simeq \Sigma^\infty BO(2)_+ \left [\tfrac{1}{2} \right ] \simeq \Sigma^\infty \bH\bP^\infty_+\left [\tfrac{1}{2} \right ]
\end{equation}
and so isomorphisms on $\bZ[\tfrac{1}{2}]$-homology
$$H_*(\mathcal{M}^{\Pin^\pm}(S_\infty);\bZ[\tfrac{1}{2}]) \cong H_*(\mathcal{M}(S_\infty);\bZ[\tfrac{1}{2}]) \cong H_*(Q(\bH\bP^\infty_+);\bZ[\tfrac{1}{2}]).$$
In particular, the rational cohomology ring is
$$H^*(\mathcal{M}^{\Pin^\pm}(S_\infty);\bQ) \cong H^*(\mathcal{M}(S_\infty);\bQ) \cong \bQ[\zeta_1, \zeta_2, \ldots]$$
where the classes $\zeta_i$ in degree $4i$ are the characteristic classes introduced by Wahl \cite{Wahl} for unoriented surface bundles. For a surface bundle $S \to E \overset{\pi} \to B$ they may be defined as the Becker--Gottlieb transfer of the $i$th power of the first Pontrjagin class of the vertical tangent bundle, that is, $\zeta_i(E) := \trf^*_\pi(p_1(T^vE)^i) \in H^{4i}(B;\bZ)$.

Ebert and the author \cite{ERW} have studied the divisibility of $\zeta_i \in H^{4i}(\mathcal{M}(S_\infty);\bZ)$, and found them to be indivisible. By the first equivalence of (\ref{eq:ZHalfHomology}) it is then clear that the classes $\zeta_i \in H^{4i}(\mathcal{M}^{\Pin^\pm}(S_\infty);\bZ)$ are divisible at most by a power of 2.

\begin{prop}
The class $\zeta_i \in H^{4i}(\mathcal{M}^{\Pin^-}(S_\infty);\bZ)$ is divisible by precisely $4^{i}$. The class $\zeta_i \in H^{4i}(\mathcal{M}^{\Pin^+}(S_\infty);\bZ)$ is divisible by $4^{i}$, and by at most $2 \cdot 4^i$.
\end{prop}
\begin{proof}
Note that the class $p_1 \in H^*(B\Pin^\pm(2);\bZ)$ is (uniquely) divisible by 4. Thus $\zeta_i = \trf^*(p_1(T^v)^i) \in H^{4i}(\mathcal{M}^{\Pin^\pm}(S_\infty);\bZ)$ is divisible by $4^i$.

To see that it is not divisible further for $\Pin^-$-structures, consider the projectivised vector bundle $\bR\bP^2 \to \bP(\gamma_3^{\Spin}) \overset{\pi}\to B\Spin(3)$. There is a vector bundle isomorphism $\pi^*\gamma_3^{\Spin} \cong T^v \bP(\gamma_3^{\Spin}) \oplus L$ where $L$ is the real line bundle characterised by the fact that it restricts to the tautological bundle on each fibre. By the Leray--Hirsch theorem, the $\bF_2$-cohomology of the total space of this fibration is the free module
$$H^*(\bP(\gamma_3^{\Spin});\bF_2) \cong H^*(B\Spin(3);\bF_2)\langle 1, x, x^2 \rangle$$
where the class $x=w_1(L)$ satisfies $x^3 = 0$, and there is a similar decomposition for integral cohomology. The total Stiefel--Whitney class of the vertical tangent bundle is
$$w(T^v) = \pi^*w(\gamma_3^{\Spin}) \cdot w(L)^{-1} = (1 + x + x^2),$$
so $w_1(T^v)=x$ and $w_2(T^v)=x^2$, and hence $T^v$ admits a $\Pin^-$-structure. Thus this bundle is formally classified by a map $B\Spin(3) \to \mathcal{M}^{\Pin^-}(S_1) \to \Omega^\infty \MT{Pin^-}{2}$.

Recall that $H^*(B\Spin(3);\bZ) = \bZ[\eta_1]$ where $4 \eta_1 = p_1$. This may be seen as $\Spin(3) = SU(2)$ so the cohomology ring is polynomial on a degree 4 generator $\eta_1 :=-c_2(\gamma_2^{SU(2)})$, and $p_1(\gamma_3^{\Spin}) = -4c_2 = 4\eta_1$. The calculation of the first Pontrjagin class follows by taking Chern classes for the identity $\gamma_3^{\Spin} \otimes_\bR \bC = \mathrm{Sym}^2\big( \gamma_2^{SU(2)} \big)$, which we learnt from \cite[Proposition 5.2.5]{Eb}. Writing $\beta$ for the Bockstein operation, we may compute\footnote{This is a general formula for the first Pontrjagin class of a rank 3 bundle which splits off a rank 1 subbundle: $p_1(\gamma_2 \oplus \gamma_1) = p_1(\gamma_2) + \beta(w_1(\gamma_2)) \cdot \beta(w_1(\gamma_1))$, which may be easily derived from the definition of Pontrjagin classes.} $\pi^*(p_1) = p_1(T^v \oplus L) = p_1(T^v) + \beta(w_1(T^v)) \cdot \beta(w_1(L))$ so that $p_1(T^v) = \pi^*(p_1) + \beta(x)^2$. The integral class $\beta(x)^2 \in H^4(\bP(\gamma_3^{\Spin});\bZ) = \bZ \langle \eta_1 \cdot 1 \rangle$ is 2-torsion and hence zero. Thus $p_1(T^v) = \pi^*(p_1)$ so
$$\zeta_i = \trf^*_\pi(p_1(T^v)^i) = \trf^*_\pi(\pi^*(p_1^i)) = \chi(\bR \bP^2) \cdot p_1^i = p_1^i = 4^i \eta_1^i \in H^{4i}(B\Spin(3) ; \bZ)$$
which is divisible by precisely $4^{i}$. Thus $\zeta_i$ is divisible by at most $4^i$.

To pass from this information about the divisibility for bundles of $\bR\bP^2$'s to divisibility in general, it is enough to note that the $\zeta_i$ are defined universally on $\Omega^\infty \MT{Pin^-}{2}$, so by the above example must also be divisible by at most $4^i$ here. The homology equivalence (\ref{eq:GroupCompletionPinMonoid}) then implies the same divisibility in the cohomology of $\mathcal{M}^{\Pin^-}(S_\infty)$.

To see that for $\Pin^+$-structures the class $\zeta_i$ is divisible by at most $2 \cdot 4^i$, note that pulling back $\zeta_i$ via $\Omega^\infty \MT{Spin}{2} \to \Omega^\infty\MT{Pin^+}{2}$ gives the class $\kappa_{2i}$, which Ebert \cite{Eb} has shown to be divisible by precisely $2^{2i+1}$.
\end{proof}

As only closed non-orientable surfaces of \textit{even} Euler characteristic admit $\Pin^+$-structures, we cannot hope to find such a structure on the projectivisation of a rank 3 vector bundle, so making precise the upper bound on the divisibility of $\zeta_i$ in this case is more difficult.

%% file: app1.tex
\section{Computing $\pi_1(\MT{Pin^\pm}{2})$}\label{sec:HomotopyOfPinPM}

We first state the structure of the $\bF_2$-cohomology algebras of $B\Pin^\pm(2)$, as algebras over the mod 2 Steenrod algebra $\mathcal{A}_2$. These may be computed from the Serre spectral sequence for the principal fibrations (\ref{eq:PinPMDefinition}). In the $\Pin^-$ case one must consult the Eilenberg--Moore spectral sequence for the fibration $B\Pin^-(2) \to BO(2) \to K(\bZ/2, 2)$ to obtain the $\mathcal{A}_2$-module structure.

\begin{prop}
There is an isomorphism of rings
$$H^*(B\Pin^+(2);\bF_2) \cong \bF_2[w_1, x_2].$$
Here, $w_1$ is the first Stiefel--Whitney class of the universal $\Pin^+$-bundle, and the higher Stiefel--Whitney classes vanish. The $\mathcal{A}_2$-module structure is determined by the ring structure and $Sq^1(x_2) = w_1 \cdot x_2$.
\end{prop}

\begin{prop}
There is an isomorphism of rings
$$H^*(B\Pin^-(2);\bF_2) \cong \bF_2[w_1, x_4]/(w_1^3).$$
Here, $w_1$ is the first Stiefel--Whitney class of the universal $\Pin^-$-bundle, $w_2(\gamma_2^{\Pin^-}) = w_1^2$, and the higher Stiefel--Whitney classes vanish. The class $x_4$ in degree 4 is the reduction of an integral class, and hence $Sq^1$ vanishes on it. Furthermore, $Sq^2$ also vanishes on it. This determines the structure as an $\mathcal{A}_2$-module.
\end{prop}

\begin{prop}\label{prop:OddCohomologyPin}
There is an isomorphism $H^*(B\Pin^\pm(2);\bZ[\frac{1}{2}]) \cong \bZ[\frac{1}{2}][p_1]$. The class $p_1$ in degree 4 is the first Pontrjagin class of the universal $\Pin^\pm$-bundle.
\end{prop}

Recall that $\MT{Pin^\pm}{2}$ is defined to be the Thom spectrum of the virtual bundle $-\gamma_2^{\Pin^\pm} \to B\Pin^\pm(2)$. This has a Thom isomorphism in $\bF_2$-cohomology, so armed with the above data, one may easily compute the $\bF_2$-cohomology of $\MT{Pin^\pm}{2}$ as an $\mathcal{A}_2$-module, and hence compute the $E_2$-page of the Adams spectral sequence in small degrees, which we include as Figure \ref{fig:AdamsSSCharts}. Note also that by (\ref{eq:ZHalfHomology}) the group $\pi_1(\MT{Pin^\pm}{2})$ has no odd torsion.

\begin{figure}[h]
\centering
\includegraphics[bb=71 582 432 716]{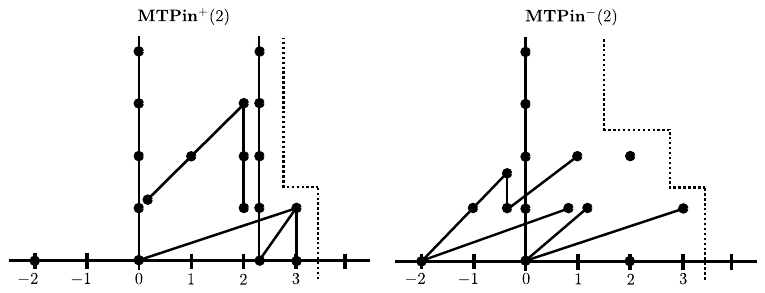}
\caption{Partial $E_2$-pages of the Adams spectral sequences converging to the 2-primary homotopy groups of the spectra $\MT{Pin^+}{2}$ and $\MT{Pin^-}{2}$ respectively. Vertical lines correspond to multiplication by $2 \in \pi_0(\mathbf{S})$, lines of slope 1 correspond to multiplication by $\eta \in \pi_1(\mathbf{S})$, and lines of slope $1/3$ correspond to multiplication by $\nu \in \pi_3(\mathbf{S})$. The diagram is complete to the left of the dotted line.}
\label{fig:AdamsSSCharts}
\end{figure}

\begin{thm}
$\pi_1(\MT{Pin^+}{2}) = \bZ/2$.
\end{thm}
\begin{proof}
The only possible differential landing in total degree 1 can be seen not to exist through the $\pi_*(\mathbf{S})$-module structure.
\end{proof}

Consider the cofibration sequence (see e.g. \cite[Proposition 3.1]{GMTW}) of spectra
$$\MT{Spin}{2} \lra \MT{Pin^-}{2} \lra \Th(-\gamma_2^{\Pin^-} \oplus \gamma_1^{\pm 1} \to B\Pin^-(2)) =:\mathbf{C}$$
where $\gamma_1^{\pm 1}$ denotes the unique non-trivial real line bundle over $B\Pin^-(2)$. The total Stiefel--Whitney class of $-\gamma_2^{\Pin^-} \oplus \gamma_1^{\pm 1}$ is $(1+w_1+w_1^2)(1+w_1) = 1$, which means that $H^*(\mathbf{C};\bF_2) \cong \Sigma^{-1}H^*(B\Pin^{-}(2)_+;\bF_2)$ as modules over the Steenrod algebra. As $x_4$ supports no non-trivial Steenrod operations, this module splits as
$$\Sigma^{-1} M \oplus \Sigma^3 M \oplus \Sigma^7 M \oplus \cdots$$
where $M = \bF_2\langle 1, w_1, w_1^2\rangle$ is a module over the Steenrod algebra with unique operation $Sq^1(w_1) = w_1^2$.

Similarly, the total Stiefel--Whitney class of $-\gamma_2^{\Spin} \to B\Spin(2)$ is 1, so there is an isomorphism $H^*(\MT{Spin}{2};\bF_2) \cong \Sigma^{-2}H^*(B\Spin(2)_+;\bF_2)$ of modules over the Steenrod algebra. This allows us to calculate the $E_2$-pages of the Adams spectral sequences converging to the 2-primary homotopy groups of these spectra, which we include as Figure \ref{fig:AdamsSSCharts2}.

\begin{figure}[h]
\centering
\includegraphics[bb=0 0 334 133]{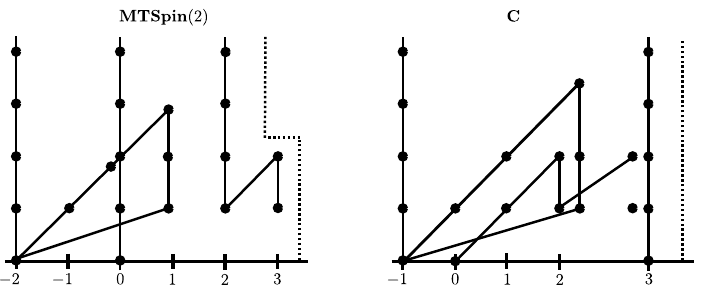}
\caption{Partial $E_2$-pages of the Adams spectral sequences converging to the 2-primary homotopy groups of the spectra $\MT{Spin}{2}$ and $\gC$ respectively. The diagram is complete to the left of the dotted line.}
\label{fig:AdamsSSCharts2}
\end{figure}

\begin{thm}
$\pi_1(\MT{Pin^-}{2}) = (\bZ/2)^3$.
\end{thm}
\begin{proof}
On the Adams $E_2$-page in total degree 1, there is a $\bF_2^2$ in filtration 1 and a $\bF_2$ in filtration 2, and there are \textit{no additive extensions}. In total degree 2 there is an $\bF_2$ in filtration 0, which could potentially kill the element in total degree 1 filtration 2. We claim there is no such differential.

Consider the long exact sequence on homotopy (modulo odd torsion) coming from the cofibration sequence 
$$\MT{Spin}{2} \lra \MT{Pin^-}{2} \lra \gC,$$
\begin{eqnarray*}
\cdots \lra  \pi_2(\gC) \lra \pi_1(\MT{Spin}{2}) \overset{\neq 0}\lra \pi_1(\MT{Pin^-}{2}) \lra (\bZ/2)^2 \overset{0}\lra \\
\bZ \oplus \bZ/2 \lra \bZ \oplus \bZ/4 \lra (\bZ/2)^2 \overset{0}\lra \bZ/2 \overset{\cong}\lra \bZ/2 \overset{0}\lra \bZ \overset{\cdot 2}\lra \bZ \lra \bZ/2.
\end{eqnarray*}
Working backwards from the end, the boundary map
$$\pi_{-1}(\gC) = \bZ \lra \pi_{-2}(\MT{Spin}{2}) = \bZ$$
is multiplication by 2, as $\pi_{-2}(\MT{Pin^-}{2}) = \bZ/2$, and hence the previous map must be zero. Thus $\pi_{-1}(\MT{Spin}{2}) \to \pi_{-1}(\MT{Pin^-}{2})$ is surjective and hence an isomorphism. The map
$$\pi_0(\MT{Spin}{2}) = \bZ \oplus \bZ/2 \lra \pi_0(\MT{Pin^-}{2}) = \bZ \oplus \bZ/4$$
has cokernel $(\bZ/2)^2$, so must be injective, and hence the previous map is zero. Note $\pi_1(\MT{Spin}{2})$ is generated by $\nu$ times a generator of $\pi_{-2}(\MT{Spin}{2})$, so it has non-trivial image in $\pi_1(\MT{Pin^-}{2})$, and the result follows.
\end{proof}